\documentclass[letterpaper,draft]{amsart} 
\usepackage[ngerman,russian,american,english]{babel}

\allowdisplaybreaks[4]

\usepackage{hyperref}

\usepackage[pdftex]{graphicx}

\usepackage{comment}

\DeclareFontFamily{U}{matha}{\hyphenchar\font45}
\DeclareFontShape{U}{matha}{m}{n}{
      <5> <6> <7> <8> <9> <10> gen * matha
      <10.95> matha10 <12> <14.4> <17.28> <20.74> <24.88> matha12
      }{}
\DeclareSymbolFont{matha}{U}{matha}{m}{n}
\DeclareFontFamily{U}{mathx}{\hyphenchar\font45}
\DeclareFontShape{U}{mathx}{m}{n}{
      <5> <6> <7> <8> <9> <10>
      <10.95> <12> <14.4> <17.28> <20.74> <24.88>
      mathx10
      }{}
\DeclareSymbolFont{mathx}{U}{mathx}{m}{n}
\DeclareMathSymbol{\obot}         {2}{matha}{"6B}
\DeclareMathSymbol{\bigobot}       {1}{mathx}{"CB}

\usepackage{mathrsfs}

\usepackage{amsfonts}
\usepackage{amsmath}
\usepackage{amssymb}
\usepackage{amsthm}
\usepackage{stmaryrd}

\usepackage{textcomp}
\usepackage{mathtext}
\usepackage{yfonts}

\makeatletter
\def\easycyrsymbol#1{\mathord{\mathchoice
  {\mbox{\fontsize\tf@size\z@\usefont{T2A}{\rmdefault}{m}{n}#1}}
  {\mbox{\fontsize\tf@size\z@\usefont{T2A}{\rmdefault}{m}{n}#1}}
  {\mbox{\fontsize\sf@size\z@\usefont{T2A}{\rmdefault}{m}{n}#1}}
  {\mbox{\fontsize\ssf@size\z@\usefont{T2A}{\rmdefault}{m}{n}#1}}
}}
\makeatother

\usepackage{algorithm}
\usepackage{algorithmic}

\usepackage{enumerate}

\usepackage{amscd}
\usepackage{tikz}
\usetikzlibrary{matrix,arrows,calc,patterns}
\usepackage{tikz-cd}
\usepackage{sidecap}
\usetikzlibrary{matrix,arrows}

\usepackage{listings}  

\usepackage{mwe} 
\usepackage{subcaption} 

\usepackage{float}
\floatstyle{boxed}
\restylefloat{table}
\restylefloat{figure}

\usepackage{nicefrac}
\newcommand{\niceonehalf}{{\nicefrac{1}{2}}}

\newcommand{\niceonefifth}{{\nicefrac{1}{5}}}

\usepackage{todonotes}

\newtheorem{theorem}{Theorem}[section]
\newtheorem{lemma}[theorem]{Lemma}

\theoremstyle{definition}

\newtheorem{example}[theorem]{Example}

\theoremstyle{remark}
\newtheorem{remark}[theorem]{Remark}

\newtheorem*{theconstants*}{Constants}

\numberwithin{equation}{section}

\newcommand{\vol}{\operatorname{vol}}
\newcommand{\eps}{{\epsilon}}

\newcommand{\inv}{{-1}}

\newcommand{\Jacobian}{\operatorname{D}}

\newcommand{\dif}{{\mathrm d}}

\newcommand{\grad}{\operatorname{grad}}

\newcommand{\curl}{\operatorname{curl}}

\newcommand{\divergence}{\operatorname{div}}
\newcommand{\diver}{\operatorname{div}}

\newcommand{\Laplace}{\Delta}

\newcommand{\diam}{{\operatorname{diam}}}

\newcommand{\dist}{\operatorname{dist}}

\newcommand{\Lip}{\operatorname{Lip}}

\newcommand{\trace}{\operatorname{tr}}

\newcommand{\Id}{\operatorname{Id}}

\newcommand{\dom}{\operatorname{dom}}

\newcommand{\rng}{\operatorname{ran}}

\newcommand{\convex}{\operatorname{convex}}

\newcommand{\cartan}{{\mathsf d}}
\newcommand{\cartanx}{{{\mathsf d}x}}

\newcommand{\onehalf}{{\frac{1}{2}}}

\newcommand{\vectornorm}[1]{\left\|#1\right\|}

\makeatletter
\newcommand\suchthat{%
 \@ifstar
  {|}
  {\mathrel{}\middle|\mathrel{}}
}
\makeatother

\newcommand\xqedhere[2]{
  \rlap{\hbox to#1{\hfil\llap{\ensuremath{#2}}}}}

\newcommand{\bbN}{{\mathbb N}}

\newcommand{\bbR}{{\mathbb R}}

\newcommand{\bbZ}{{\mathbb Z}}

\newcommand{\calC}{{\mathcal C}}

\newcommand{\calH}{{\mathcal H}}

\newcommand{\calJ}{{\mathcal J}}

\newcommand{\calL}{{\mathcal L}}

\newcommand{\calP}{{\mathcal P}}

\newcommand{\calT}{{\mathcal T}}
\newcommand{\calU}{{\mathcal U}}

\newcommand{\frakD}{{\mathfrak D}}

\newcommand{\frakH}{{\mathfrak H}}

\newcommand{\frakL}{{\mathfrak L}}

\newcommand{\vecf}{{\vec f}}

\newcommand{\vecn}{{\vec n}}

\newcommand{\vecu}{{\vec u}}

\usepackage{url}

\title[Smoothed Projections and Mixed BC]{Smoothed Projections and\\ Mixed Boundary Conditions}

\author[Martin W.\ Licht]{Martin W.\ Licht}

\address{Martin Licht, UCSD Department of Mathematics, 9500 Gilman Drive MC0112, La Jolla, CA 92093-0112}

\email{mlicht@ucsd.edu}

\thanks{This research was supported by the European Research Council through 
the FP7-IDEAS-ERC Starting Grant scheme, project 278011 STUCCOFIELDS}

\subjclass[2000]{65N30, 58A12}

\keywords{finite element exterior calculus, smoothed projection, partial boundary conditions, mixed boundary conditions}

\begin{document}

\def\justbeingincluded{justbeingincluded}


 %
 
\begin{abstract}
 Mixed boundary conditions are introduced to finite element exterior calculus.
 We construct smoothed projections from Sobolev de~Rham complexes onto finite element de~Rham complexes 
 which
 commute with the exterior derivative,
 preserve homogeneous boundary conditions along a fixed boundary part,
 and
 satisfy uniform bounds for shape-regular families of triangulations
 and bounded polynomial degree.
 The existence of such projections implies sta\-bi\-li\-ty and quasi-optimal convergence of mixed finite element methods 
 for the Hodge Laplace equation with mixed boundary conditions. 
 In addition, we prove  
 the density of smooth differential forms 
 in Sobolev spaces of differential forms over weakly Lipschitz domains with partial boundary conditions. 
\end{abstract}

\maketitle


%
%

\section{Introduction}
In this article, we address the numerical analysis of the Hodge Laplace equation
when \emph{mixed boundary conditions} are imposed. 
Here, we speak of mixed boundary conditions 
when essential boundary conditions are imposed on one part of the boundary,
while natural boundary conditions are imposed on the complementary boundary part.
Special cases are the Poisson equation 
with mixed Dirichlet and Neumann boundary conditions \cite{necas2011direct},
and the vector Laplace equation with mixed tangential and normal boundary conditions \cite{fernandes1997magnetostatic}. 
It is known in the theory of partial differential equations 
that the Hodge Laplace equation with mixed boundary conditions 
arises from Sobolev de~Rham complexes with \emph{partial boundary conditions} \cite{jakab2009regularity,GMM}.
These are composed of spaces of Sobolev differential forms in which boundary conditions 
are imposed only on a part of the boundary (corresponding to the essential boundary conditions).

Moving towards the numerical analysis of mixed finite element methods 
for the Hodge Laplace equation with mixed boundary conditions,
we adopt the framework of \emph{finite element exterior calculus} (FEEC, \cite{AFW1}).
Smoothed projections from Sobolev de~Rham complexes onto finite element de~Rham complexes 
play a central role in the a priori error analysis 
within FEEC. 
The literature provides the corresponding smoothed projections in the special cases of 
either fully essential 
or fully natural boundary conditions,
but the general case of mixed boundary conditions has remained open \cite{AFW1,AFW2}.
In order to enable the abstract Galerkin theory of FEEC
for the general case of mixed boundary conditions,
we need a smoothed projection that preserves partial boundary conditions.

Constructing such a smoothed projection is the main contribution of this article.
The abstract Galerkin theory of Hilbert complexes 
then shows quasi-optimal convergence of a large class of mixed finite element methods. 
\\

Mixed boundary conditions for vector-valued partial differential equations 
are a non-trivial topic, and even more so in numerical analysis.
We outline the topic and prior research, 
starting with mixed boundary conditions for the Poisson problem. 
Suppose that $\Omega$ is a bounded connected Lipschitz domain
with outward unit normal field $\vecn$ along $\partial\Omega$.
Let $\Gamma_D \subseteq \partial\Omega$ 
be a subset of the boundary with positive surface measure
and let $\Gamma_N := \partial\Omega \setminus \Gamma_D$. 
Given a function $f$, the Poisson problem
with mixed boundary conditions asks for the solution $u$ of  
\begin{gather}
 \label{intro:poissonproblem}
 - \Delta u = f, 
 \quad 
 u_{|\Gamma_{D}} = 0, 
 \quad 
 \nabla u_{|\Gamma_{N}} \cdot \vec n = 0.
\end{gather}
Here, we impose a homogeneous \emph{Dirichlet boundary condition} along $\Gamma_{D}$
and a homogeneous \emph{Neumann boundary condition} along $\Gamma_{N}$.
If $f \in L^{2}(\Omega)$, 
then a weak formulation 
characterizes the solution as the unique minimizer of the energy 
\begin{gather}
 \label{intro:poissonenergy}
 \calJ( v )
 := 
 \onehalf \int_{\Omega}
 | \grad v |^{2}
 \;\dif x
 -
 \int_{\Omega}
 f v
 \;\dif x
\end{gather}
over $H^{1}(\Omega,\Gamma_D)$, the subspace of $H^{1}(\Omega)$
whose members satisfy the (essential) Dirichlet boundary condition along $\Gamma_D$.
The well-posedness of this variational problem follows by a Friedrichs inequality
with partial boundary conditions \cite{necas2011direct}.
Moreover, the compactness of the embedding $H^{1}(\Omega,\Gamma_D) \rightarrow L^{2}(\Omega)$
is crucial in proving the compactness of the solution operator.
A typical finite element method seeks a discrete approximation of $u$ by minimizing $\calJ$
over a space of Lagrange elements in $H^{1}(\Omega,\Gamma_D)$.
The discussion of this Galerkin method is standard \cite{Braess2007},
but still we cannot approach the Poisson problem with mixed boundary conditions
by the current means of FEEC due to the lack of a smoothed projection.
\\

The natural generalization to vector-valued problems in three dimensions
is given by the vector Laplace equation with mixed boundary conditions.
This equation appears in electromagnetism or fluid dynamics.
The analysis of this vector-valued partial differential equation, 
however, is considerably more complex.
Given the vector field $\vec f$,
we seek a vector field $\vec u$ that solves  
\begin{gather}
 \label{intro:vectorlaplacian:core}
 \curl \curl \vec u - \grad \diver \vec u = \vec f 
\end{gather}
over the domain $\Omega$. 
Moreover we assume that $\partial\Omega = \Gamma_T \cup \Gamma_N$
is an essentially disjoint partition of the boundary into relatively open subsets;
geometric details are discussed later in this article.
The boundary conditions on $\vec u$ are 
\begin{gather}
 \label{intro:vectorlaplacian:boundary}
 \vec u_{|\Gamma_{T}} \times \vec n = 0
 , 
 \quad
 (\diver \vec u)_{|\Gamma_{T}} = 0
 ,
 \quad 
 \vec u_{|\Gamma_{N}} \cdot \vec n = 0
 ,
 \quad
 (\curl \vec u)_{|\Gamma_{N}} \times \vec n  = 0
 .
\end{gather}
Here we impose homogeneous \emph{tangential boundary conditions} on $\vec u$ along a boundary part $\Gamma_{T}$,
and homogeneous \emph{normal boundary conditions} on $\vec u$ along the complementary boundary part $\Gamma_{N}$.
When $\vecf \in L^{2}(\Omega,\bbR^{3})$, 
then a variational formulation seeks the solution by minimizing the energy functional
\begin{gather}
 \label{intro:vectorlaplaceenergy}
 \calJ( \vec v )
 := 
 \onehalf 
 \int_{\Omega} 
  | \divergence \vec v |^{2}
  + 
  | \curl \vec v |^{2}
 \;\dif x
 -
 \int_{\Omega} 
  \vecf \cdot \vec v 
 \;\dif x
\end{gather}
over the space $H(\divergence,\Omega,\Gamma_N) \cap H(\curl,\Omega,\Gamma_T)$.
Here $H(\divergence,\Omega,\Gamma_N)$ is the subspace of $H(\divergence,\Omega)$
satisfying normal boundary conditions along $\Gamma_N$, and $H(\curl,\Omega,\Gamma_T)$
is the subspace of $H(\curl,\Omega)$ satisfying tangential boundary conditions 
along $\Gamma_T$.

The additional complexity in comparison to the scalar-valued case 
begins with the correct definition of tangential and normal boundary conditions
in a setting of low regularity 
\cite{weck1974maxwell,buffa2002traces,buffa2003trace,weck2004traces,GMM}.
When non-mixed boundary conditions are imposed, 
so that either $\Gamma_T = \emptyset$ or $\Gamma_T = \partial\Omega$,
then 
Rellich-type compact embeddings 
$H(\divergence,\Omega,\Gamma_N) \cap H(\curl,\Omega,\Gamma_T) \rightarrow L^{2}(\Omega,\bbR^{3})$
and vector-valued Poincar\'e-Friedrichs inequalities have been known for a long time 
\cite{weber1980local,picard1984elementary,witsch1993remark,costabel2010bogovskiui}.
Mixed boundary conditions in vector analysis, however, 
have only recently been addressed systematically in pure analysis
\cite{jochmann1997compactness,kuhn1999maxwellgleichung,jochmann1999regularity,alonso2003unique,bauer2015maxwell}.

Additional difficulties arise in numerical analysis.
Minimizing \eqref{intro:vectorlaplaceenergy} over a finite element subspace
of $H(\divergence,\Omega,\Gamma_N) \cap H(\curl,\Omega,\Gamma_T)$ generally does not lead 
to a consistent finite element method \cite{costabel1991coercive,AFW2}.
But mixed finite element methods,
which introduce either $\divergence \vecu$ or $\curl \vecu$ as auxiliary variables,
have been studied with great success
\cite{bossavit1998computational,hiptmair2002finite,monk2003finite,demkowicz2006computing}.
Mixed boundary conditions for the vector Laplace equation, however,
have not yet received much attention in numerical analysis (but see \cite{schoberl2005multilevel,gopalakrishnan2011partial}).
A mixed finite element method for the vector Laplace equation 
with mixed boundary conditions only incorporates 
the essential boundary conditions along $\Gamma_T$ into the finite element space.
\\

We attend particularly to a phenomenon that 
significantly affects the theoretical and numerical analysis of the vector Laplace equation
but remains absent in the scalar-valued theory:
the presence of non-trivial harmonic vector fields in $H(\divergence,\Omega,\Gamma_N) \cap H(\curl,\Omega,\Gamma_T)$.
Specifically, let $\vec\calH(\Omega,\Gamma_{T},\Gamma_{N})$
denote the subspace of $H(\divergence,\Omega,\Gamma_N) \cap H(\curl,\Omega,\Gamma_T)$
whose members have vanishing curl and vanishing divergence.
This space has physical relevance:
in fluid dynamics, for example, those vector fields describe the incompressible irrotational flows 
that satisfy given boundary conditions.
In the case of non-mixed boundary conditions,
their dimension corresponds to topological properties of the domain \cite{mitrea2002finite},
and in particular that dimension is zero on contractible domains.
But in the case of mixed boundary conditions, 
this dimension depends on the topology of both the domain $\Omega$ and the boundary part $\Gamma_T$.
Thus $\vec\calH(\Omega,\Gamma_{T},\Gamma_{N})$ may have positive dimension
if $\Gamma_T$ has a sufficiently complicated topology even if $\Omega$ itself is contractible 
\cite{kress1971kombiniertes,GMM}.
This dimension can be calculated exactly from a given triangulation of $\Omega$ and $\Gamma_T$.
In a finite element method, the subspace $\vec\calH(\Omega,\Gamma_{T},\Gamma_{N})$ 
must be approximated by discrete harmonic fields, i.e., the kernel of the finite element vector Laplacian.
\\

It is instructive to study these partial differential equations in a unified manner using the calculus of differential forms. 
Both the Poisson problem and the vector Laplace equation with mixed boundary conditions 
are special cases of the Hodge Laplace equation with mixed boundary conditions.
The Hodge Laplace equation has been studied extensively over Sobolev spaces of differential forms
\cite{weck1974maxwell,bruening1992hilbert,scott1995Lp,kuhn1999maxwellgleichung,mitrea2001dirichlet,mitrea2002traces,axelsson2004hodge,mitrea2004sharp,weck2004traces},
but the case of mixed boundary conditions has been a recent subject of research 
\cite{GMM,jakab2009regularity}.
The theoretical foundation are de~Rham complexes with partial boundary conditions.
Let $L^{2}\Lambda^{k}(\Omega)$ be the space of differential forms with square-integrable coefficients,
and let $H\Lambda^{k}(\Omega,\Gamma_T)$ be the subspace of $L^{2}\Lambda^{k}(\Omega)$
whose members have exterior derivatives in $L^{2}\Lambda^{k+1}(\Omega)$
and satisfy boundary conditions along $\Gamma_T$.
We consider the differential complex 
\begin{align}
 \label{intro:sobolevderhamcomplex}
 \begin{CD}
  \dots
  @>\cartan>>
  H\Lambda^k(\Omega,\Gamma_T)
  @>\cartan>>
  H\Lambda^{k+1}(\Omega,\Gamma_T)
  @>\cartan>>
  \dots
 \end{CD}
\end{align}
where $\cartan$ denotes the exterior derivative. 
The widely studied special cases $\Gamma_T = \emptyset$ and $\Gamma_T = \partial\Omega$
correspond to either imposing 
essential boundary conditions 
either nowhere or along the entire boundary.

The calculus of differential forms has attracted interest as a unifying framework for mixed finite element methods
\cite{hiptmair2002finite,AFW1,AFW2,falk2014local,arnold2013finite,demlow2014posteriori}.
The numerical analysis of mixed finite element methods 
for the Hodge Laplace equation can be formulated in terms of finite element de~Rham complexes,
which mimic the differential complex \eqref{intro:sobolevderhamcomplex} on a discrete level.
Finite element exterior calculus is a formalization of that approach. 
But whereas the special cases of non-mixed boundary conditions are standard applications,
mixed boundary conditions (corresponding to more general choices of $\Gamma_T$) have not been addressed in this context yet.

We outline the corresponding finite element de~Rham complexes.
We let $\calT$ be a triangulation of $\Omega$ that also contains a triangulation $\calU$ of $\Gamma_T$.
Moreover, we let $\calP\Lambda^{k}(\calT)$ denote a space 
of piecewise polynomial differential forms in $H\Lambda^{k}(\Omega)$
of the type $\calP_r\Lambda^{k}(\calT)$
or $\calP_r^{-}\Lambda^{k}(\calT)$ as described in \cite{AFW1,afwgeodecomp}.
Then the finite element space with essential boundary conditions is 
$\calP\Lambda^{k}(\calT,\calU) := \calP\Lambda^{k}(\calT) \cap H\Lambda^{k}(\Omega,\Gamma_T)$.
In this article we classify a family of finite element de~Rham complexes 
\begin{align}
 \label{intro:FEderhamcomplex}
 \begin{CD}
  \dots 
  @>\cartan>>
  \calP\Lambda^{k  }(\calT,\calU)
  @>\cartan>>
  \calP\Lambda^{k+1}(\calT,\calU)
  @>\cartan>>
  \dots
 \end{CD}
\end{align}
that feature these essential boundary conditions
and guide the construction of stable mixed finite element methods,
completely analogous to the classification of finite element de~Rham complexes
in the case of non-mixed boundary conditions \cite{AFW1}.

To relate the continuous and discrete levels
and enable the abstract Galerkin theory of FEEC (\cite{AFW2}),
we need a projection $\pi^{k} : H\Lambda^{k  }(\Omega,\Gamma_T) \rightarrow \calP\Lambda^{k  }(\calT,\calU)$ 
onto the finite element space 
that 
is uniformly $L^{2}$ bounded 
and 
commutes with the differential operator. 
In particular, the following diagram commutes:
\begin{align}
 \label{intro:commutingdiagram}
 \begin{CD}
  \dots
  @>\cartan>>
  H\Lambda^{k  }(\Omega,\Gamma_T)
  @>\cartan>>
  H\Lambda^{k+1}(\Omega,\Gamma_T)
  @>\cartan>>
  \dots
  \\
  @. @V{\pi^{k}}VV @V{\pi^{k+1}}VV @.
  \\
  \dots
  @>\cartan>>
  \calP\Lambda^{k  }(\calT,\calU)
  @>\cartan>>
  \calP\Lambda^{k+1}(\calT,\calU)
  @>\cartan>>
  \dots
 \end{CD}
\end{align}
Given such a projection, we obtain a priori convergence results 
for mixed finite element methods \cite{AFW2}.
A specific example are the smoothed projections 
in finite element exterior calculus  
for non-mixed boundary conditions
\cite{AFW1,christiansen2008smoothed,licht2016smoothed}.
\\

The main contribution of this article 
are smoothed projections $\pi^{k}$
for de~Rham complexes with partial boundary conditions.
In particular, we bound their operator norm in terms of the mesh quality 
and the finite element polynomial degree.

Continuing the research in \cite{licht2016smoothed},
we assume minimal geometric regularity 
and conduct our construction over \emph{weakly Lipschitz domains},
which is a class of domains generalizing classical (strongly) Lipschitz domains.
In order to define mixed boundary conditions,
we partition the boundary of the domain into two complementary parts 
on which we impose essential or natural boundary conditions, respectively.
We assume only minimal regularity for the boundary partition.
This choice of geometric ambient has favorable properties.
On the one hand, 
the class of weakly Lipschitz domains is broad enough to contain (strongly) Lipschitz domains
and a large class of three-dimensional polyhedral domains that fail to be strongly Lipschitz,
such as the crossed brick domain \cite[Figure 3.1]{monk2003finite}.
On the other hand, 
many analytical results that are known for strongly Lipschitz domains,
such as the Rellich embedding theorem for differential forms,
still hold true over weakly Lipschitz domains (see \cite{axelsson2004hodge,bauer2015maxwell}). 
Notably, restricting to strongly Lipschitz domains 
would not simplify the mathematical derivations in this article.
In continuity with \cite{licht2016smoothed},
we define the smoothed projections over differential forms with coefficients in $L^{p}$ spaces
for $p \in [1,\infty]$ and consider the $W^{p,q}$ classes of differential forms
\cite{gol1982differential}.
\\

We give an outline of the stages that compose the smoothed projection.
Let $u \in L^{2}\Lambda^{k}(\Omega)$. 
First, an operator $E^{k} : L^{2}\Lambda^{k}(\Omega) \rightarrow L^{2}\Lambda^{k}(\Omega^{\rm e})$ 
extends $u$ over a neighborhood $\Omega^{\rm e}$ of $\Omega$.
The basic idea is extending the differential form by reflection,
but along $\Gamma_T$ we extend it by zero over a ``bulge'' attached to the domain.
$E^{k}$ commutes with the exterior derivative on $H\Lambda^{k}(\Omega,\Gamma_T)$.
Next, we construct a distortion $\frakD : \Omega^{\rm e} \rightarrow \Omega^{\rm e}$
which moves a neighborhood of the bulge into the latter but is the identity away from the bulge.
We locally control the amount of distortion.
The pullback $\frakD^{\ast} E^{k} u$ of $E^{k} u$ along $\frakD$ vanishes then in a neighborhood of $\Gamma_T$
and commutes with the exterior derivative.
Subsequently, we apply a regularization operator 
$R^{k} : L^{2}\Lambda^{k}(\Omega^{\rm e}) \rightarrow C^{\infty}\Lambda^{k}(\overline\Omega)$ 
which produces a smoothing of $\frakD^{\ast} E^{k} u$ 
that still vanishes in a neighborhood of $\Gamma_T$. 
This is based on the idea of taking the convolution with a smooth bump function.
In our case, however, the smoothing radius is locally controlled.

The canonical finite element interpolant 
$I^{k} : C^{\infty}\Lambda^{k}(\Omega) \rightarrow \calP\Lambda^{k}(\calT)$
is then applied to the regularized differential form $R^{k} \frakD^{\ast} E^{k} u$.
Since the latter vanishes near $\Gamma_T$, the resulting differential form is an element of $\calP\Lambda^{k}(\calT,\calU)$.
In combination, this yields an operator $Q^{k} : L^{2}\Lambda^{k}(\Omega) \rightarrow \calP\Lambda^{k}(\calT,\calU)$
that commutes with the exterior derivative on $H\Lambda^{k}(\Omega,\Gamma_T)$
and satisfies uniform $L^{2}$ bounds. 
But $Q^{k}$ is generally not idempotent.
To enforce idempotence, we bound the interpolation error over the finite element space
and apply the ''Sch\"oberl trick'' \cite{schoberl2005multilevel}.
This delivers the desired smoothed projection.
In particular, we prove the following main result. 

\begin{theorem}
 Let $\Omega \subseteq \bbR^{n}$ be a bounded weakly Lipschitz domain,
 and let $\Gamma_T \subseteq \partial\Omega$ be an admissible boundary patch 
 (in the sense of Section~\ref{sec:geometry}).
 Let $\calT$ be a simplicial triangulation of $\Omega$
 that contains a simplicial triangulation $\calU$ of $\Gamma_T$,
 and let \eqref{intro:FEderhamcomplex} be a finite element de~Rham complex 
 as in finite element exterior calculus \cite{AFW1}
 with essential boundary conditions along $\Gamma_T$.
 Then there exist bounded linear projections
 $\pi^{k} : L^{2}\Lambda^{k}(\Omega) \rightarrow \calP\Lambda^{k}(\calT,\calU)$
 such that the following diagram commutes:
 \begin{align}
  \begin{CD}
   H\Lambda^{0}(\Omega,\Gamma_T)
   @>\cartan>>
   H\Lambda^{1}(\Omega,\Gamma_T)
   @>\cartan>>
   \cdots
   @>\cartan>>
   H\Lambda^{n}(\Omega,\Gamma_T)
   \\
   @V{\pi^{0}}VV @V{\pi^{1}}VV @. @V{\pi^{n}}VV
   \\
   \calP\Lambda^{0}(\calT,\calU)
   @>\cartan>> 
   \calP\Lambda^{1}(\calT,\calU)
   @>\cartan>>
   \cdots 
   @>\cartan>> 
   \calP\Lambda^{n}(\calT,\calU).
  \end{CD}
 \end{align}
 Moreover, $\pi^{k} u = u$ for $u \in \calP\Lambda^{k}(\calT,\calU)$.
 The $L^{2}$ operator norm of $\pi^{k}$ is uniformly bounded in terms of 
 the maximum polynomial degree of \eqref{intro:FEderhamcomplex},
 the shape measure of the triangulation,
 and geometric properties of $\Omega$ and $\Gamma_T$.
\end{theorem}

The calculus of differential forms has emerged as a unifying language 
for finite element methods for problems in vector analysis 
\cite{hiptmair2002finite} and commuting projections play an important role in that context.
A bounded projection operator that commutes with the exterior derivative 
up to a controllable error was derived in \cite{christiansen2007stability}.
A bounded commuting projection operator for the de~Rham complex without 
boundary conditions has been derived in \cite{AFW1} in the case of quasi-uniform triangulations,
which was subsequently generalized in \cite{christiansen2008smoothed}
to shape-uniform triangulations and de~Rham complexes with full boundary conditions.
The ideas in those contributions were extended to smoothed projections 
over weakly Lipschitz domains in \cite{licht2016smoothed},
which we take as our point of departure.
The existence of a smoothed projection that respects partial boundary conditions
has been an unproven conjecture in \cite{bonizzoni2014moment}.
A local bounded interpolation operator was presented in \cite{schoberl2008posteriori}
in the language of classical vector analysis;
this result was generalized to differential forms in \cite{demlow2014posteriori},
and a variant of this operator that preserves partial boundary condition
was given in \cite{gopalakrishnan2011partial}.
A local commuting projection is developed in \cite{falk2014local}.
We refer to \cite{christiansen2011topics,ern2016mollification,gopalakrishnan2012commuting}
for further approaches on the topic.
\\

In addition to this research in numerical analysis, we contribute a result to functional analysis.
Specifically, we prove that smooth differential forms over a weakly Lipschitz domain $\Omega$
which vanish near $\Gamma_T$ are dense in $H\Lambda^{k}(\Omega,\Gamma_T)$ for $p,q \in [1,\infty)$.
When $\Omega$ is a strongly Lipschitz domain and $\Gamma_D \subseteq \partial\Omega$ is a suitable boundary patch, 
then the density of $C^{\infty}(\overline\Omega) \cap H^{1}(\Omega,\Gamma_D)$ in $H^{1}(\Omega,\Gamma_D)$ (see \cite{doktor1973density,doktor2006density}) 
and analogous density results for differential forms with partial boundary conditions (see \cite{jakab2009regularity})
are known. 
The following generalization to weakly Lipschitz domains, however, has been not available in the literature yet.

\begin{lemma}
 Let $\Omega$ be a bounded weakly Lipschitz domain 
 and let $\Gamma_T$ be an admissible boundary patch
 (in the sense of Section~\ref{sec:geometry}).
 Then the smooth differential $k$-forms in $C^{\infty}\Lambda^{k}(\overline\Omega)$
 that vanish near $\Gamma_T$
 constitute a dense subset of $H\Lambda^{k}(\Omega,\Gamma_T)$. 
\end{lemma}

The remainder of this article is structured as follows.
In Section~\ref{sec:differentialforms}, we review the calculus of differential forms. 
In Section~\ref{sec:geometry}, we introduce the geometric setting and the extension operator.
In Section~\ref{sec:distortion}, we devise the distortion mapping. 
In Section~\ref{sec:mollification}, we combine these constructions 
to obtain the mollification operator. 
In Section~\ref{sec:projection}, we devise the smoothed projection.
Finally, Section~\ref{sec:application} outlines applications to the convergence theory
of finite element methods.

%
%

\section{Lipschitz Analysis and Differential Forms}
\label{sec:differentialforms}

In this section we recall background material in several fields of analysis.
This includes a summary of Section 3 of \cite{licht2016smoothed}
with an additional discussion of differential forms over domains  
that satisfy homogeneous boundary conditions along subsets of the boundary.
We draw from sources in Lipschitz analysis \cite{luukkainen1977elements},
geometric measure theory \cite{federer2014geometric},
the calculus of differential forms over domains \cite{LeeSmooth},
and differential forms with coefficients in $L^{p}$ spaces \cite{gol1982differential,GMM}.

\subsection{Analytical Preliminaries}

We recall some basic notions. 
Unless mentioned otherwise,
we let every subset $A \subseteq \bbR^{n}$, $n \in \bbN_0$,
be equipped with the canonical Euclidean norm $\|\cdot\|$.
For any set $U \subseteq \bbR^{n}$ and $r > 0$ we write $B_r(U)$
for the closed Euclidean $r$-neighborhood of $U$ 
and set $B_r(x) := B_r(\{x\})$.
Moreover, $\vol^{m}(A)$ is the $m$-dimensional exterior Hausdorff measure of any subset $A \subset \bbR^{m}$.

Let $U \subseteq \bbR^{n}$ and $V \subseteq \bbR^{m}$,
and let $\Phi : U \rightarrow V$ be a mapping.
For a subset $A \subseteq U$ we define $\Lip(\Phi,A) \in [0,\infty]$
as the minimal member of $[0,\infty]$
such that 
\begin{align*}
 \forall x, y \in A :
 \left\| \Phi(x) - \Phi(y) \right\| \leq \Lip(\Phi,A) \| x - y \|
 .
\end{align*}
We call $\Lip(\Phi,A)$ the \emph{Lipschitz constant} of $\Phi$ over $A$
and we simply write $\Lip(\Phi) := \Lip(\Phi,A)$ if $A$ is understood.
We say that $\Phi$ is Lipschitz if $\Lip(\Phi,U) < \infty$.
If $\Phi : U \rightarrow V$ is invertible,
then we call $\Phi$ \emph{bi-Lipschitz} 
if both $\Phi$ and $\Phi^{\inv}$ are Lipschitz.
We call $\Phi$ a \emph{LIP embedding} if $\Phi : U \rightarrow \Phi(U)$
is bijective and locally bi-Lipschitz.

\subsection{Differential Forms}

Let $U \subseteq \bbR^n$ be an open set.
We let $M(U)$ denote the space of Lebesgue-measurable functions over $U$. 
For $k \in \bbZ$ we let $M\Lambda^{k}(U)$ be the space of 
differential $k$-forms over $U$ with coefficients in $M(U)$. 
For $u \in M\Lambda^{k}(U)$ and $v \in M\Lambda^{l}(U)$
we let $u \wedge v \in M\Lambda^{k+l}(U)$
denote the \emph{exterior product} of $u$ and $v$.
We let $C^{\infty}\Lambda^{k}(U)$ be the space of smooth differential forms over $U$, 
we let $C^{\infty}\Lambda^{k}(\overline U)$ denote the space of smooth differential forms over $U$ 
that are restrictions of members of $C^{\infty}\Lambda^{k}(\bbR^{n})$,
and we let $C^{\infty}_{c}\Lambda^{k}(U)$ be the subspace of $C^{\infty}\Lambda^{k}(\overline U)$ 
whose members have compact support.

We let $\cartanx^1, \ldots, \cartanx^n \in M\Lambda^{1}(U)$
be the constant $1$-forms that represent the $n$ coordinate directions.
We let $\langle u, v \rangle \in M(U)$ denote the pointwise $\ell^{2}$ product 
of two measurable differential $k$-forms $u, v \in M\Lambda^{k}(U)$
(see Equation (3.3) in \cite{licht2016smoothed} for a definition).
Accordingly, we let $|u| = \sqrt{\langle u, u \rangle} \in M(U)$
be the pointwise $\ell^{2}$ norm of $u \in M(U)$.
The \emph{Hodge star operator} 
$\star : M\Lambda^{k}(U) \rightarrow M\Lambda^{n-k}(U)$
is a linear mapping that is uniquely defined by the identity
\begin{align}
 \label{math:hodgestar}
 u \wedge \star v = \langle u, v \rangle \; \cartanx^{1} \wedge \dots \wedge \cartanx^{n}
 \quad 
 u, v \in M\Lambda^{k}(U).
\end{align}
We recall the \emph{exterior derivative} 
$\cartan : C^{\infty}\Lambda^{k}(U) \rightarrow C^{\infty}\Lambda^{k+1}(U)$
of smooth differential forms. 
More generally,
if $u \in M\Lambda^{k}(U)$ and $w \in M\Lambda^{k+1}(U)$ such that 
\begin{align}
 \label{math:distributionalexteriorderivative}
 \int_U w \wedge v
 =
 (-1)^{k+1} \int_U u \wedge \cartan v,
 \quad 
 v \in C_c^{\infty}\Lambda^{n-k-1}(U),
\end{align}
then we call $w$ the \emph{weak exterior derivative} of $u$
and write $\cartan u := w$.
Note that $w$ is unique up to equivalence almost everywhere in $U$,
and that the weak exterior derivative of $u \in C^{\infty}\Lambda^{k}(\overline U)$
agrees with the (strong) exterior derivative almost everywhere.
The Hodge star enters the definition of the \emph{exterior codifferential},
which is a differential operator given (in the strong sense) by 
\begin{align*}
 \delta : 
 C^{\infty}\Lambda^{k}(U) \rightarrow C^{\infty}\Lambda^{k-1}(U),
 \quad 
 u \mapsto (-1)^{k(n-k)+1} \star \cartan \star u.
\end{align*}
A weak exterior codifferential can be defined analogously.

\subsection{$W^{p,q}$ differential forms}
We work with differential forms 
whose coefficients are contained in Lebesgue spaces. 
We let $L^{p}(U)$ denote the Lebesgue space over $U$ with exponent $p \in [1,\infty]$
and let $L^{p}\Lambda^{k}(U)$ be the Banach space of differential $k$-forms
with coefficients in $L^{p}(U)$,
together with the norm
\begin{align*}
 \left\| u \right\|_{L^{p}\Lambda^{k}(U)}
 :=
 \big\|
  |u| 
 \big\|_{L^{p}(U)},
 \quad
 u \in L^{p}\Lambda^{k}(U),
 \quad 
 p \in [1,\infty].
\end{align*}
For $p, q \in [1,\infty]$, 
we let $W^{p,q}\Lambda^{k}(U)$ be the Banach space of differential $k$-forms in $L^{p}\Lambda^{k}(U)$
with weak exterior derivative in $L^{q}\Lambda^{k+1}(U)$.
Its natural norm is 
\begin{align}
 \| u \|_{W^{p,q}\Lambda^{k}(U)}
 :=
 \| u \|_{L^{p}\Lambda^{k}(U)}
 +
 \| \cartan u \|_{L^{q}\Lambda^{k+1}(U)}.
\end{align}
Since the exterior derivative of an exterior derivative is zero,
we have for every choice of $p,q,r \in [1,\infty]$ the inclusion 
\begin{gather*}
 \cartan W^{p,q}\Lambda^{k}(U) \subseteq W^{q,r}\Lambda^{k+1}(U)
 .
\end{gather*}

\begin{example}
 The space $W^{1,1}\Lambda^{k}(U)$ contains all integrable differential $k$-forms over $U$
 with integrable weak exterior derivative. 
 If $U$ is bounded, then all other spaces $W^{p,q}\Lambda^{k}(U)$ embed into $W^{1,1}\Lambda^{k}(U)$. 
 The space $W^{2,2}\Lambda^{k}(U)$, consisting of square-integrable differential $k$-forms 
 with square-integrable exterior derivatives,
 has a Hilbert space structure equivalent to its Banach space structure 
 and is often written $H\Lambda^{k}(U)$ in the literature.
 The space $W^{\infty,\infty}\Lambda^{k}(U)$ comprises those essentially bounded differential $k$-forms
 whose exterior derivative is essentially bounded; these are called \emph{flat differential forms}
 in geometric measure theory \cite{federer2014geometric,whitney2012geometric}.
\end{example}

We are particularly interested in spaces 
of differential forms that satisfy homogeneous boundary conditions along a subset $\Gamma$ of the boundary $\partial U$.
We call these \emph{partial boundary conditions}.
Following Definition 3.3 of \cite{GMM}, we discuss boundary conditions 
with the intuition that a differential form with weak exterior derivative 
satisfies homogeneous boundary conditions along the boundary part $\Gamma$
if and only if its extension by zero still has a weak exterior derivative along that boundary part.

Formally, 
when $\Gamma \subseteq \partial U$ is a relatively open subset of $\partial U$, 
then we define the space $W^{p,q}\Lambda^k(U,\Gamma)$ as 
the subspace of $W^{p,q}\Lambda^{k}(U)$ whose members adhere to the following condition: 
we have $u \in W^{p,q}\Lambda^k(U,\Gamma)$ if and only if
for all $x \in \Gamma$ there exists $r > 0$ such that
\begin{align}
 \label{math:ana:bc}
 \int_{U \cap B_{r}(x)} u \wedge \cartan v
 =
 \int_{U \cap B_{r}(x)} {\cartan u} \wedge v,
 \quad 
 v \in C^{\infty}_c\Lambda^{n-k-1}\left(\mathring B_{r}(x)\right).
\end{align}
The definition implies that $W^{p,q}\Lambda^k(U,\Gamma)$
is a closed subspace of $W^{p,q}\Lambda^{k}(U)$,
and hence a Banach space of its own.
We also say that $u \in W^{p,q}\Lambda^k(U,\Gamma)$
satisfies partial boundary conditions along $\Gamma$.
One consequence of the definition is 
\begin{align}
 \cartan W^{p,q}\Lambda^{k}(U,\Gamma) \subseteq W^{q,r}\Lambda^{k+1}(U,\Gamma),
 \quad 
 p,q,r \in [1,\infty].
\end{align}
In other words, 
a differential form which satisfies boundary conditions along $\Gamma$
has an exterior derivative satisfying boundary conditions along $\Gamma$.

\begin{remark}
 For example, $W^{p,q}\Lambda^{k}(U,\partial U)$ is the subspace of $W^{p,q}\Lambda^{k}(U)$
 whose member's extension to $\bbR^{n}$ by zero gives a member of $W^{p,q}\Lambda^{k}(\bbR^{n})$.
 If the domain boundary is sufficiently regular, 
 then an equivalent notion of homogeneous boundary conditions 
 uses generalized trace operators \cite{mitrea2002traces,weck2004traces}.
 This article does not address inhomogeneous boundary conditions.
\end{remark}

We finish this section with some results on the behavior of differential forms
under bi-Lipschitz coordinate changes.
Suppose that $U, V \subseteq \bbR^{n}$ are connected open sets
and let $\Phi : U \rightarrow V$ be a bi-Lipschitz mapping.
The pullback of $u \in M\Lambda^{k}(V)$ along $\Phi$ 
is the differential form $\Phi^{\ast}u \in M\Lambda^{k}(U)$
given at $x \in U$ by 
\begin{align}
 \label{math:pullback:definition}
 \Phi ^{\ast} u_{|x}( \nu_1, \dots, \nu_k )
 :=
 u_{|\Phi (x)} ( \Jacobian\Phi_{|x} \cdot \nu_1, \dots, \Jacobian\Phi_{|x} \cdot \nu_k ),
 \quad
 \nu_{1},\dots,\nu_{k} \in \bbR^{n}.
\end{align}
The pullback commutes with the exterior derivative in the sense that 
$\cartan \Phi^{\ast} u = \Phi^{\ast} \cartan u$ 
whenever $u \in M\Lambda^{k}(U)$ has a weak exterior derivative.
The pullback along bi-Lipschitz mappings also preserves the $L^{p}$ classes of differential forms.
We henceforth write $1/\infty := 0$.
One can show that for $p \in [1,\infty]$ and $u \in L^{p}\Lambda^{k}(V)$ 
we have $\Phi^{\ast} u \in L^{p}\Lambda^{k}(U)$ with 
\begin{align}
  \label{math:pullbackestimate}
  \begin{split}
  \| \Phi ^{\ast} u \|_{L^{p}\Lambda^{k}(U)}
  &\leq 
  \left\| \Jacobian\Phi \right\|^{k}_{L^{\infty}(U,\bbR^{n\times n})}
  \left\| \Jacobian\Phi^{\inv} \right\|^{\frac{n}{p}}_{L^{\infty}(V,\bbR^{n\times n})}
  \| u \|_{L^{p}\Lambda^{k}(V)}
  .
  \end{split}
\end{align}
Moreover, the $W^{p,q}$ classes of differential forms are preserved under pullbacks:
for $p,q \in [1,\infty]$ and $u \in W^{p,q}\Lambda^{k}(V)$
we have $\Phi^{\ast} u \in W^{p,q}\Lambda^{k}(U)$.
This follows immediately from \eqref{math:pullbackestimate}
via the commutativity of the pullback and the exterior derivative.

%
%

\section{Boundary Partitions of Weakly Lipschitz Domains}
\label{sec:geometry}

In this section we discuss weakly Lipschitz domains and boundary partitions with minimal regularity assumptions.
We provide the geometric background and introduce commuting extension operators that take into account partial boundary conditions. 
Our discussion of weakly Lipschitz domains and boundary partitions is based on \cite{GMM}.

\subsection{Weakly Lipschitz Domains and Boundary Partitions}
Let $\Omega \subseteq \bbR^{n}$ be a domain.
We call $\Omega$ a \emph{weakly Lipschitz domain}
if every $x \in \partial\Omega$ has a closed neighborhood
$U_x \subseteq \bbR^{n}$ for which there exists a bi-Lipschitz mapping
$\varphi_x : U_x \rightarrow [-1,1]^n$ such that $\varphi_x(x) = 0$ and 
\begin{subequations}
\label{math:weaklylipschitzdomain}
\begin{align}
 \label{math:weaklylipschitzdomain:condition1}
 \quad \varphi_x( \Omega   \cap U_x ) &= [-1,1]^{n-1} \times [-1,0)
 , \\
 \label{math:weaklylipschitzdomain:condition2}
 \quad \varphi_x( \partial\Omega \cap U_x ) &= [-1,1]^{n-1} \times \{0\}
 , \\
 \label{math:weaklylipschitzdomain:condition3}
 \quad \varphi_x( \overline\Omega^c \cap U_x ) &= [-1,1]^{n-1} \times (0,1]
 .
\end{align}
\end{subequations}

\begin{remark}
 In other words, $\Omega$ is a weakly Lipschitz domain 
 if its boundary can be flattened locally by a bi-Lipschitz mapping. 
 For example, every Lipschitz domain (a domain whose boundary 
 can be written locally as the graph of a Lipschitz function)
 is also a weakly Lipschitz domain. The converse is generally false,
 and a well-known counter example are the ``crossed bricks'' \cite[p.39]{monk2003finite}. 
 But every polyhedral domain in $\bbR^{3}$ is a weakly Lipschitz domain \cite[Theorem 4.1]{licht2016smoothed}.
\end{remark}

Next we introduce the geometric background for the discussion of boundary conditions. 
We assume that $\Gamma_{T} \subseteq \partial\Omega$ is a relatively open subset of the boundary.
We write $\Gamma_{N} := \partial\Omega \setminus \overline{\Gamma_T}$ for the complementary relatively open boundary part
and we write $\Gamma_{I} := \overline{\Gamma_{T}} \cap  \overline{\Gamma_{N}}$ for the interface 
between the boundary parts.
We call $\Gamma_T$ an \emph{admissible boundary patch} 
if for every $x \in \Gamma_I$ there exists a bi-Lipschitz mapping 
$\varphi_x : U_x \rightarrow [-1,1]$
such that \eqref{math:weaklylipschitzdomain} holds and such that additionally 
\begin{subequations}
\label{math:admissiblepartition} 
\begin{align}
 \label{math:admissiblepartition:condition1}
 \varphi_x( \Gamma_T   \cap U_x )
 &=
 [-1,1]^{n-2} \times [-1,0) \times \{0\}
 , \\
 \label{math:admissiblepartition:condition2}
 \varphi_x( \Gamma_I \cap U_x )
 &=
 [-1,1]^{n-2} \times \{0\} \times \{0\}
 , \\
 \label{math:admissiblepartition:condition3}
 \varphi_x( \Gamma_N \cap U_x )
 &=
 [-1,1]^{n-2} \times (0,1] \times \{0\}
 .
\end{align}
\end{subequations}
We also call the triple $\left( \Gamma_T,\Gamma_I,\Gamma_N \right)$ an \emph{admissible boundary partition}. 
Note that 
$\Gamma_T$ is an admissible boundary patch 
if and only if 
$\Gamma_N$ is an admissible boundary patch.

\begin{remark}
 A weakly Lipschitz domain is the interior of a locally flat $n$-dimensional Lipschitz submanifold of $\bbR^{n}$ with boundary.
 In particular, $\partial\Omega$ is a locally flat Lipschitz submanifold of dimension $n-1$ without boundary.
 The tuple $(\Gamma_T,\Gamma_I,\Gamma_N)$ being an admissible boundary partition 
 means that $\Gamma_T$ and $\Gamma_N$ are the interior of locally flat Lipschitz submanifolds of dimension $n-1$ 
 of $\partial\Omega$ with common boundary $\Gamma_I := \partial\Gamma_T = \partial\Gamma_N$.
 In turn, $\Gamma_I$ is a Lipschitz submanifold of dimension $n-2$ without boundary. 
 This is in accordance with Definition 3.7 of \cite{GMM}, when Remark 3.2 in that reference is taken into account.
\end{remark}

\subsection{Commuting Extension Operators}
In the remainder of this section we construct a commuting extension operator. 
The basic idea comprises two steps.
First, any differential form over $\Omega$
is extended by zero onto a bulge domain $\Upsilon$ attached to $\Omega$ along $\Gamma_T$.
Second, the thus extended differential form over this enlarged domain 
is extended again to an even larger domain $\Omega^{\rm e}$
via a reflection along the boundary. 
\\

We begin with some geometric definitions and results. 
A tubular neighborhood of $\Omega$ with Lipschitz regularity is obtained from Theorem~2.3 of \cite{licht2016smoothed}. 
Specifically, 
there exists a LIP embedding $\Psi^{0} : \partial\Omega \times [-1,1] \rightarrow \bbR^{n}$
satisfying 
\begin{subequations}
\label{math:omegacollar}
\begin{gather}
 \label{math:omegacollar:boundary}
 \forall x \in \partial\Omega : \Psi^{0}\left( x, 0 \right) = x, 
 \\
 \label{math:omegacollar:claimedinclusions}
 \Psi^{0}\left( \partial\Omega \times [-1,0) \right) \subseteq \Omega,
 \quad
 \Psi^{0}\left( \partial\Omega \times ( 0,1] \right) \subseteq \overline\Omega^{c}
 .
\end{gather}
\end{subequations}
We then introduce the auxiliary domains 
\begin{gather}
 \label{math:bulgeddomains}
 \Upsilon := \Psi^{0}\left( \Gamma_T \times [0,\niceonehalf) \right),
 \quad 
 \Omega^{\rm b} := \Omega \cup \Gamma_T \cup \Upsilon.
\end{gather}
We think of $\Upsilon$ as a \textit{bulge} attached to the domain $\Omega$ along $\Gamma_T$,
which results in the combined domain $\Omega^{\rm b}$.
It is easily verified that both $\Upsilon$ and $\Omega^{\rm b}$ 
are again weakly Lipschitz domains;
see also Figure~\ref{fig:bulgeddomain} for a visualization 
of the construction.

The fact that $\Omega^{\rm b}$ is a weakly Lipschitz domain
allows us to apply Theorem~2.3 of \cite{licht2016smoothed} again 
to construct a tubular neighborhood with Lipschitz regularity:
we obtain a LIP embedding 
$\Psi^{\rm b} : \partial\Omega^{\rm b} \times [-1,1] \rightarrow \bbR^{n}$
that satisfies 
\begin{subequations}
\label{math:bulgecollar}
\begin{gather}
 \label{math:bulgecollar:boundary}
 \forall x \in \partial\Omega^{\rm b} : \Psi^{\rm b}\left( x, 0 \right) = x, 
 \\
 \label{math:bulgecollar:claimedinclusions}
 \Psi^{\rm b}\left( \partial\Omega^{\rm b} \times [-1,0) \right) \subseteq \Omega^{\rm b},
 \quad
 \Psi^{\rm b}\left( \partial\Omega^{\rm b} \times ( 0,1] \right) \subseteq \overline{ \Omega^{\rm b} }^{c}
 .
\end{gather}
\end{subequations}
We now define the additional auxiliary domains 
\begin{gather}
 \calC\Omega^{\rm b} 
 :=
 \Psi^{\rm b}( \partial\Omega \times (-1,1) )
 ,
 \quad 
 \Omega^{\rm e} 
 :=
 \Omega^{\rm b} 
 \;\cup\;
 \Psi^{\rm b}\left( \partial\Omega^{\rm b} \times [0,1) \right)
 ,
\end{gather}
which by Theorem~2.3 of \cite{licht2016smoothed}
can be assumed to be weakly Lipschitz domains without loss of generality.
We are now in the position to define the main result of this section:
a commuting extension operator preserving partial boundary conditions.

\begin{theorem}
 \label{prop:extensionoperator}
 There exists a bounded linear operator 
 \begin{gather*}
  E^{k} : L^{p}\Lambda^{k}(\Omega) \rightarrow L^{p}\Lambda^{k}(\Omega^{\rm e}),
  \quad 
  p \in [1,\infty], 
 \end{gather*}
 such that $E^{k} u$ vanishes over $\Upsilon$ for every $u \in L^{p}\Lambda^{k}(\Omega)$. 
 Moreover, for all $p,q \in [1,\infty]$ and $u \in W^{p,q}\Lambda^{k}(\Omega,\Gamma_T)$
 we have
 \begin{gather*}
  E^{k} u \in W^{p,q}(\Omega^{\rm e}),
  \quad 
  \cartan E^{k} u = E^{k+1} \cartan u
  .
 \end{gather*}
 Lastly, 
 there exist $L_{E} \geq 1$ and $C_{E} \geq 1$
 such that for every $p \in [1,\infty]$, 
 every $\delta > 0$,
 and every measurable $A \subseteq \overline\Omega$
 we have 
 \begin{align}
  \label{prop:extensionoperator:localbound}
  \| E^{k} u \|
  _{L^{p}\Lambda^{k}\left( B_{\delta}(A) \cap \Omega^{\rm e} \right)}
  \leq 
  \left( 1 + C_{E}^{ k + \frac{n}{p} } \right)
  \| u \|
  _{L^{p}\Lambda^{k}\left( B_{\delta L_{E}}(A) \cap \overline\Omega \right)}
  ,
  \quad 
  u \in L^{p}\Lambda^{k}(\Omega)
  .
 \end{align}
\end{theorem}

\begin{proof}
 For any $u \in M\Lambda^{k}(\Omega)$ we let $E_{0}^{k} u \in M\Lambda^{k}(\Omega^{\rm b})$
 denote the extension of $u$ by zero to $\Omega^{\rm b}$. 
 Moreover, we let $E_{r}^{k} : M\Lambda^{k}(\Omega^{\rm b}) \rightarrow M\Lambda^{k}(\Omega^{\rm e})$
 be the extension operator based on reflection
 that is obtained by applying the results 
 of Subsection~7.1 in \cite{licht2016smoothed} to our domain $\Omega^{\rm b}$
 and the tubular neighborhood described by $\Psi^{\rm b}$. 

 We set $E^{k} := E_{r}^{k} E_{0}^{k}$.
 We first observe that $E^{k} u$ vanishes over $\Upsilon$ for every $u \in M\Lambda^{k}(\Omega)$, 
 since by construction $E_{0}^{k}$ vanishes over $\Upsilon$ 
 and $E_{r}^{k} E_{0}^{k} u$ agrees with $E_{0}^{k} u$ over $\Omega^{\rm b}$.
 Next, by Lemma~7.1 of \cite{licht2016smoothed} 
 we have for every $p \in [1,\infty]$ a bounded mapping 
 $E^{k}_{r} : L^{p}\Lambda^{k}(\Omega^{\rm b}) \rightarrow L^{p}\Lambda^{k}(\Omega^{\rm e})$. 
 In combination, for every $p \in [1,\infty]$
 we have a bounded mapping $E^{k} : L^{p}\Lambda^{k}(\Omega) \rightarrow L^{p}\Lambda^{k}(\Omega^{\rm e})$. 
 
 Assume that $p,q \in [1,\infty]$ and $u \in W^{p,q}\Lambda^{k}(\Omega,\Gamma_T)$. 
 Then $E_{0}^{k} u \in W^{p,q}(\Omega^{\rm b})$ with $E_{0}^{k} \cartan u = \cartan E_{0}^{k} u$ by definition. 
 Moreover, 
 we have $E_{r}^{k} E_{0}^{k} u \in W^{p,q}(\Omega^{\rm e})$ 
 with $E_{r}^{k} E_{0}^{k} \cartan u = E_{r}^{k} \cartan E_{0}^{k} u = \cartan E_{r}^{k} E_{0}^{k} u$
 by Lemma~7.4 of \cite{licht2016smoothed}. 
 
 Finally, \eqref{prop:extensionoperator:localbound} follows with a combination of Lemma~7.2 of \cite{licht2016smoothed}
 applied to the extension operator $E_{r}$
 together with the fact that $E_{0}$ is an extension by zero. 
\end{proof}

\begin{remark} 
 A similar extension operator can be found in \cite{schoberl2008posteriori},
 whose technical details are elaborated in \cite{gopalakrishnan2011partial}
 for the special case of (strongly) Lipschitz domains 
 and boundary partitions with a piecewise $C^{1}$-interfaces.
 We consider a more general geometric setting in this contribution
 and use weaker assumptions. 
\end{remark}

\begin{figure}[t]
 ${}$
 \\
 \begin{center}
  \begin{tikzpicture}[scale=1.0]
    
    \definecolor{hellgrau}{rgb} {0.90,0.90,0.90} 
  
    \coordinate (LOB) at ( -2.5,  1.5, 0.0);
    \coordinate (LLB) at ( -3.5,  0.0, 0.0);
    \coordinate (LUB) at ( -2.5, -1.5, 0.0);
    \coordinate (MOB) at (  0.0,  0.5, 0.0);
    \coordinate (MUB) at (  0.0, -0.5, 0.0);
    \coordinate (ROB) at (  2.9,  1.3, 0.0);
    \coordinate (RRB) at (  4.0,  0.0, 0.0);
    \coordinate (RUB) at (  2.9, -1.3, 0.0);
    
    \coordinate (LOA) at ( -2.5,  2.3, 0.0);
    \coordinate (LLA) at ( -4.0,  0.0, 0.0);
    \coordinate (LUA) at ( -2.9, -1.9, 0.0);
    \coordinate (MOA) at (  0.0,  1.5, 0.0);
    \coordinate (MUA) at (  0.0, -1.5, 0.0);
    \coordinate (ROA) at (  3.0,  2.3, 0.0);
    \coordinate (RRA) at (  4.5,  0.0, 0.0);
    \coordinate (RUA) at (  3.3, -1.7, 0.0);
    
    \coordinate (LOH) at ( -2.5,  1.9, 0.0);
    \coordinate (MOH) at (  0.0,  1.0, 0.0);
    \coordinate (ROH) at (  3.0,  1.8, 0.0);
    
    \coordinate (LOI) at ( -2.0,  0.7, 0.0);
    \coordinate (LLI) at ( -3.0,  0.0, 0.0);
    \coordinate (LUI) at ( -2.0, -0.7, 0.0);
    \coordinate (MOI) at (  0.0,  0.2, 0.0);
    \coordinate (MUI) at (  0.0, -0.2, 0.0);
    \coordinate (ROI) at (  2.4,  0.7, 0.0);
    \coordinate (RRI) at (  3.5,  0.0, 0.0);
    \coordinate (RUI) at (  2.4, -0.7, 0.0);
    
    \draw [ultra thick,black,fill=hellgrau]
          (MOB)
          to[out=0   ,in=180,dashed]
          (ROB)
          to[out=0   ,in=90]
          (RRB)
          to[out=-90 ,in=0] 
          (RUB)
          to[out=-180,in=0] 
          (MUB)
          to[out=-180,in=0] 
          (LUB)
          to[out=-180,in=-90] 
          (LLB)
          to[out=  90,in=180] 
          (LOB)
          to[out=   0,in=180] 
          (MOB)
          ;
    
    \draw [black]
          (MOA)
          to[out=0   ,in=180]
          (ROA)
          to[out=0   ,in=90]
          (RRA)
          to[out=-90 ,in=0] 
          (RUA)
          to[out=-180,in=0] 
          (MUA)
          to[out=-180,in=0] 
          (LUA)
          to[out=-180,in=-90] 
          (LLA)
          to[out=  90,in=180] 
          (LOA)
          to[out=   0,in=180] 
          (MOA)
          ;
    
    \draw [black]
          (MOI)
          to[out=0   ,in=180]
          (ROI)
          to[out=0   ,in=90]
          (RRI)
          to[out=-90 ,in=0] 
          (RUI)
          to[out=-180,in=0] 
          (MUI)
          to[out=-180,in=0] 
          (LUI)
          to[out=-180,in=-90] 
          (LLI)
          to[out=  90,in=180] 
          (LOI)
          to[out=   0,in=180] 
          (MOI)
          ;
    
    \draw [ultra thick,black,dashed,fill=hellgrau,pattern=crosshatch, pattern color=hellgrau]
          (MOB)
          to[out=   0,in=180]
          (ROB)
          to[out=  90,in=-90]
          (ROH)
          to[out= 180,in=0] 
          (MOH)
          to[out= 180,in=0] 
          (LOH)
          to[out=- 90,in=90] 
          (LOB)
          to[out=-  0,in=180] 
          (MOB)
          ;
    
    \draw [ultra thick,black]
          (LOB)
          to[out=   0,in=180]
          (MOB)
          to[out=   0,in=180]
          (ROB)
          ;
    
    \draw [line width = 1.9pt,hellgrau,dashed] 
          (LOB)
          to[out=   0,in=180]
          (MOB)
          to[out=   0,in=180]
          (ROB)
          ;

  \end{tikzpicture}
\end{center}
  
 \caption{
 Domain $\Omega$ (gray) with a bulge $\Upsilon$ (shaded) attached.
 The thick line is the boundary part $\Gamma_N$ of the original domain,
 and the contact line between $\Omega$ and $\Upsilon$ is the boundary part $\Gamma_T$. 
 The thinner lines inside and outside of the domain 
 indicate the inner and outer boundaries, respectively, of a tubular neighborhood $\Psi^{0}$. 
 }
 \label{fig:bulgeddomain}
\end{figure}
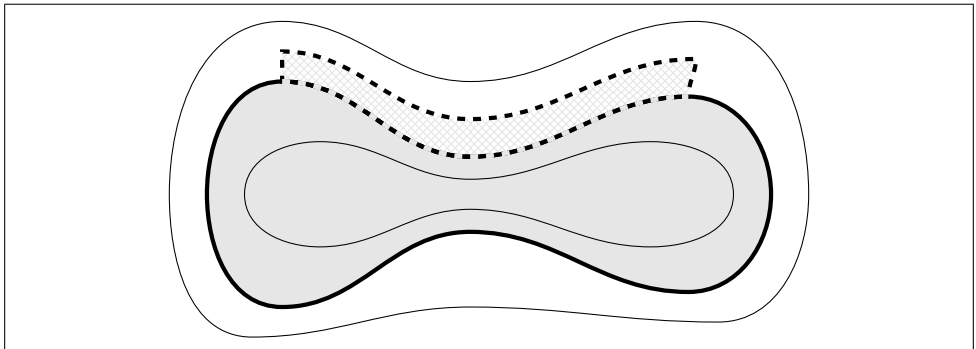

%
%

\section{Distortion of Domain Boundaries}
\label{sec:distortion}

In this section we discuss a geometric result 
that enters the construction of the smoothed projection 
but which is also of independent interest.
The basic idea is as follows: 
given a domain $\Upsilon \subseteq \bbR^{n}$,
we search for a homeomorphism of $\bbR^{n}$
that moves $\partial\Upsilon$ into $\Upsilon$
and that is the identity outside of a neighborhood of $\partial\Upsilon$.
Moreover, we want to locally control how far the homeomorphism
moves the boundary into the domain.
Specifically, we prove the following result.

\begin{theorem}
 \label{prop:distortionmapping}
 Let $\Upsilon \subseteq \bbR^{n}$ be a bounded weakly Lipschitz domain.
 There exist $\eps_{D} > 0$ and $L_{D} > 0$ 
 such that for any non-negative function $\varrho : \bbR^{n} \rightarrow \bbR$
 satisfying 
 \begin{align}
  \label{prop:distortionmapping:condition}
  \Lip(\varrho,\bbR^{n})
  < 
  \eps_{D},
  \quad 
  \max\limits_{x \in \bbR^{n}} \varrho(x) 
  < 
  \eps_{D}
  ,
 \end{align}
 there exists a bi-Lipschitz mapping $\frakD_{\varrho} : \bbR^{n} \rightarrow \bbR^{n}$
 with the following properties.
 \begin{subequations}
 \label{prop:distortionmapping:properties} 
 We have 
 \begin{gather}
  \label{prop:distortionmapping:estimates} 
  \Lip(\frakD_{\varrho}) \leq L_{D} \left( 1 + \Lip(\varrho) \right),
  \quad 
  \Lip(\frakD^{\inv}_{\varrho}) \leq L_{D} \left( 1 + \Lip(\varrho) \right).
 \end{gather}
  We have 
  \begin{align}
   \label{prop:distortionmapping:preservebulge}
   \frakD_{\varrho}(\Upsilon)
   \subseteq
   \Upsilon
   .
  \end{align}
  For all $x \in \bbR^{n}$ we have 
  \begin{align}
   \label{prop:distortionmapping:shiftbound}
   \left\| x - \frakD_{\varrho}(x) \right\|
   \leq
   L_{D}
   \varrho(x)
   . 
  \end{align}
  For all $x \in \bbR^{n}$ we have 
  \begin{align}
   \label{prop:distortionmapping:identityaway}
   \dist\left( x, \partial\Upsilon \right)
   \geq 
   L_{D}
   \varrho(x)
   \;\;\implies\;\;
   \frakD_{\varrho}(x) = x
   .
  \end{align}
  For all $x \in \partial\Upsilon$ we have 
  \begin{gather}
   \label{prop:distortionmapping:intobulge}
   \frakD_{\varrho}\left( B_{ \varrho(x) / L_{D} }(x) \right)
   \subseteq 
   \Upsilon
   .
  \end{gather}
  \end{subequations}
\end{theorem}

\begin{remark}
 We discuss the meaning and application of Theorem~\ref{prop:distortionmapping} before we give the proof. 
 The mapping $\frakD_{\varrho}$ is a distortion of $\bbR^{n}$ which moves $\Upsilon$ into itself. 
 The function $\varrho$ controls the amount of distortion near $\Upsilon$. 
 The distortion $\frakD_{\varrho}$ contracts a neighborhood of $\Upsilon$ into the domain.
 Specifically, we interpret the properties \eqref{prop:distortionmapping:properties} in the following manner. 
 Property~\eqref{prop:distortionmapping:preservebulge} formalizes that the distortion 
 moves $\partial\Upsilon$ into $\Upsilon$; in particular, $\Upsilon$ is mapped into itself.
 Property~\eqref{prop:distortionmapping:identityaway} formalizes that the homeomorphism is the identity outside of a neighborhood of $\partial\Upsilon$.
 By Property~\eqref{prop:distortionmapping:shiftbound} the amount of distortion is locally bounded by $\varrho$,
 and Property~\eqref{prop:distortionmapping:intobulge} formalizes that the distortion is proportional to $\varrho$ near the boundary.
\end{remark}

\begin{proof}[Proof of Theorem~\ref{prop:distortionmapping}]
 Since $\Upsilon \subseteq \bbR^{n}$ is a bounded weakly Lipschitz domain,
 we can apply Theorem~2.3 of \cite{licht2016smoothed} to deduce
 the existence of a LIP embedding 
 \begin{gather*}
  \Xi : \partial\Upsilon \times [-1,1] \rightarrow \bbR^{n}
 \end{gather*}
 such that $\Xi(x,0) = x$ for $x \in \partial\Upsilon$
 and such that 
  $\Xi\left( \partial\Upsilon, [0,1] \right) 
  \subset 
  \overline\Upsilon$.
 In particular, there exist constants $c_{\Xi}, C_{\Xi} > 0$ such that
 \begin{gather*}
  \left\| \Xi(x_1,t_1) - \Xi(x_2,t_2) \right\|
  \leq 
  C_{\Xi}
  \sqrt{ \| x_1 - x_2 \|^{2} + | t_1 - t_2 |^{2} },
  \\
  \sqrt{ \| x_1 - x_2 \|^{2} + | t_1 - t_2 |^{2} }
  \leq 
  c_{\Xi} \left\| \Xi(x_1,t_1) - \Xi(x_2,t_2) \right\|
 \end{gather*}
 for $x_1, x_2 \in \partial\Upsilon$ and $t_1, t_2 \in [-1,1]$.
 We note in particular that 
 \begin{align*}
  c_{\Xi}^{\inv}
  \Lip( \varrho ) 
  \leq 
  \Lip\left( \varrho \Xi \right)
  \leq
  C_{\Xi} \Lip(\varrho)
  .
 \end{align*}
 For $\alpha \in [0,\niceonefifth]$ we consider the parametrized mappings 
 \begin{gather*}
  \zeta_{\alpha} : 
  [-1,1] \rightarrow [-1,1],
  \quad 
  t 
  \mapsto 
  \int_{-1}^{t}
  1 + \chi_{[-2\alpha,\alpha]} - \frac{2}{3}\chi_{[\alpha,3\alpha]} \;\dif{\lambda} - 1,
  \\
  \zeta_{\alpha}^{\inv} : 
  [-1,1] \rightarrow [-1,1],
  \quad 
  t 
  \mapsto
  \int_{-1}^{s}
  1 + \frac{2}{3} \chi_{[-2\alpha,\alpha]} - 2 \chi_{[3\alpha,4\alpha]} \;\dif{\lambda} - 1
  ,
 \end{gather*}
 where $\chi_{I}$ is the indicator function of the interval $I \subseteq [-1,1]$. 
 As the notation already suggests, these two mappings are mutually inverse for $\alpha$ fixed.
 We easily see that they are strictly monotonically increasing,
 and that their Lipschitz constants are uniformly bounded for $\alpha \in [0,\niceonefifth]$.
 In particular $\zeta_{\alpha}$ and $\zeta_{\alpha}^{\inv}$ are bi-Lipschitz.
 Moreover, we observe for, say, $\alpha \in [0,\niceonefifth]$ that 
 \begin{gather}
  \label{math:zetaproperty:eins}
  \zeta_{\alpha}(t) = \zeta_{\alpha}^{\inv}(t) = t,
  \quad 
  t \notin [ -2\alpha, 4\alpha ],
  \\
  \label{math:zetaproperty:zwei}
  \zeta_{\alpha}\left( [-\alpha,\alpha] \right) 
  =
  [ \alpha, 3\alpha ]
  .
 \end{gather}
 We now write $\zeta(t;\alpha) = \zeta_{\alpha}(t)$ and $\zeta^{\inv}(t;\alpha) = \zeta_{\alpha}^{\inv}(t)$
 for $(t,\alpha) \in [-1,1] \times [0,\niceonefifth]$.
 Assume from now on that 
 \begin{gather*}
  \max\limits_{x \in \bbR^{n}} \varrho(x) < \niceonefifth,
  \quad 
  \Lip( \varrho, \bbR^{n} ) < \min\left\{ 1, \Lip( \Xi )^{\inv} \right\}.
 \end{gather*}
 This implies that $\Lip( \varrho \Xi ) < 1$. 
 We define homeomorphisms 
 \begin{gather*}
  \frakD_{\varrho} : \bbR^{n} \rightarrow \bbR^{n},
  \quad 
  \frakD_{\varrho}^{\inv} : \bbR^{n} \rightarrow \bbR^{n},
 \end{gather*}
 in the following manner.
 Assume that $x \in \bbR^{n}$.
 If there exist $x_0 \in \partial\Upsilon$ and $t \in [-1,1]$
 such that $x = \Xi(x_0,t)$, then we set  
 \begin{gather*}
  \frakD_{\varrho}( x )
  :=
  \Xi\left( x_0, t' \right),
  \quad 
  t' := \zeta\left( t; \varrho(x_0) / 8 \right)
  ,
  \\
  \frakD_{\varrho}^{\inv}( x )
  :=
  \Xi\left( x_0, t'' \right),
  \quad 
  t'' := \zeta^{\inv}\left( t; \varrho(x_0) / 8 \right)
  .
 \end{gather*}
 Otherwise, we set $\frakD_{\varrho}(x) := x$.
 It follows from the construction that $\frakD_{\varrho}$ and $\frakD_{\varrho}^{\inv}$
 are bi-Lipschitz and mutually inverse. 
 In particular, \eqref{prop:distortionmapping:estimates} is implied by  
 \begin{gather*}
  \Lip\left(\frakD_{\varrho}\right) 
  \leq
  1 + c_{\xi}^{\inv} C_{\Xi} \left( 1 + \Lip\left(\zeta\right) \Lip(\varrho) \right),
  \\ 
  \Lip\left(\frakD_{\varrho}^{\inv}\right)
  \leq
  1 + c_{\xi}^{\inv} C_{\Xi} \left( 1 + \Lip\left(\zeta^{\inv}\right) \Lip(\varrho) \right).
 \end{gather*}
 The construction shows that \eqref{prop:distortionmapping:preservebulge} holds,
 since $\frakD_{\varrho}$ maps $\Xi\left( \partial\Upsilon, [0,1] \right)$ into itself.
 Moreover, $\frakD_{\varrho}$ and $\frakD_{\varrho}^{\inv}$ act like the identity
 outside of $\Xi\left( \partial\Upsilon, [-1,1] \right)$.
 
 Let us assume for the remainder of this proof that 
 that $x = \Xi(x_0,t)$ for $x_0 \in \partial\Upsilon$ and $t \in [-1,1]$.
 Using \eqref{math:zetaproperty:eins}, we see that 
 $\frakD_{\varrho}(x) \neq x$
 implies 
 $|t| \leq \varrho(x_0) / 2$, 
 and so
 \begin{align*}
  \left| \varrho(x_0) - \varrho(x) \right|
  =
  \left| \varrho( \Xi(x_0,0) ) - \varrho( \Xi(x_0,t) ) \right|
  \leq
  \frac{ \Lip(\varrho\Xi) \varrho(x_0) }{ 2 } 
  \leq 
  \frac{ \varrho(x_0) }{ 2 }.
 \end{align*}
 This implies that $\varrho(x_0) \leq 2\varrho(x)$. 
 By the definition of $\zeta$ and $\frakD_{\varrho}$ we then see  
 \begin{align*}
  \left| x - \frakD_{\varrho}(x) \right|
  \leq 
  \Lip(\Xi)
  \left| t - \zeta\left(t; \dfrac{\varrho(x_0)}{ 8 } \right) \right| 
  \leq 
  \frac{6}{8}
  \Lip(\Xi) \varrho(x_0) 
  \leq 
  \frac{3}{2}
  \Lip(\Xi) \varrho(x) 
  , 
 \end{align*}
 proving \eqref{prop:distortionmapping:shiftbound}.
 Furthermore, using  we note that $x \neq \frakD_{\varrho}(x)$ implies 
 \begin{align*}
  \dist\left( x, \partial\Upsilon \right)
  \leq 
  \| x - x_{0} \|
  \leq 
  \Lip(\Xi) |t|
  \leq 
  \dfrac{\Lip(\Xi)}{2} \varrho(x_0)
  \leq 
  \Lip(\Xi) \varrho(x)
  . 
 \end{align*}
 Now \eqref{prop:distortionmapping:identityaway} follows because clearly 
 $x = \frakD_{\varrho}(x)$ is implied by  
 \begin{gather*}
  \dist(x,\Gamma_T) \geq \Lip(\Xi) \varrho(x).
 \end{gather*}
 It remains to prove \eqref{prop:distortionmapping:intobulge}.
 We define $A \subseteq \partial\Upsilon \times [-1,1]$ by 
 \begin{gather*}
  A 
  :=
  \left( B_{ \varrho(x_0) / 8 }(x_0) \cap \partial\Upsilon \right)
  \times 
  \left( -7 \varrho(x_0) / 64, 7 \varrho(x_0) / 64 \right)
  .
 \end{gather*}
 If $y \in \partial\Upsilon$ 
 with $\| x_0 - y \| \leq \varrho(x_0) / 8$,
 then $\| \varrho(x_0) - \varrho(y) \| \leq \varrho(x_0)/8$
 since we assume $\Lip(\varrho) < 1$.
 In particular $\varrho(y) \geq 7\varrho(x_0)/8$ follows.
 Via \eqref{math:zetaproperty:eins} we thus find $\frakD_{\varrho}( \Xi(A) ) \subseteq \Upsilon$.
 We observe that $A$ contains a ball around $x_0$ of radius $7\varrho(x_0) / 64$ in $\partial\Upsilon \times [-1,1]$.
 Whence $\Xi(A)$ contains a ball around $x_0$ of radius $c_{\Xi}^{\inv} 7 \varrho(x_0) / 64$.
 This shows \eqref{prop:distortionmapping:intobulge}.
 The proof is complete. 
\end{proof}

%
%

\section{Mollification with Partial Boundary Conditions}
\label{sec:mollification}

In this section we construct a mollification operator 
for differential forms on the weakly Lipschitz domain $\Omega$
that respects partial boundary conditions along the admissible boundary patch $\Gamma_T$.
As in the classical mollification operator 
we let define the mollified differential form at each point 
by averaging the coefficients in a small neighborhood of that point. 
But as a technical difference to the classical construction 
we let the mollification radius vary over the space. 
\\

The \emph{standard mollifier} is the function 
\begin{align}
 \label{math:standardmollifier}
 \mu : \bbR^{n} \rightarrow \bbR,
 \quad
 y
 \mapsto \left\{ \begin{array}{lcl}
  C_{\mu} \exp\left( - ( 1 - \|y\|^{2} )^{\inv} \right)
  & \text{ if } & \|y\| \leq 1, 
  \\
  0
  & \text{ if } & \|y\| > 1,
 \end{array}\right.
\end{align}
where $C_{\mu} > 0$ is chosen such that $\mu$ has unit integral.
Note that $\mu$ is smooth and has compact support in $B_1(0)$.
We also write $\mu_{r}(y) := r^{-n}\mu(y/r)$. 

The canonical smoothing operator over the Euclidean space is defined by convolution with the standard mollifier. 
We consider a generalization where the smoothing radius is locally controlled by a function.
Throughout this section we let $\varrho : \bbR^{n} \rightarrow \bbR$
be a non-negative smooth function that assumes a positive minimum over $\overline\Omega$.
We introduce the mapping 
\begin{align}
 \label{math:mollification_transform}
 \Phi_{\varrho} :
 \bbR^{n} \times B_1(0)
 \rightarrow
 \bbR^n,
 \quad
 (x,y)
 \mapsto 
 x + \varrho(x) y
 .
\end{align}
Regarding the second variable as a parameter, we have a family of mappings 
\begin{align*}
 \Phi_{\varrho,y} : 
 \bbR^{n} \rightarrow \bbR^{n},
 \quad 
 x \mapsto \Phi_{\varrho}( x, y ).
\end{align*}
It is easily seen that $\Phi_{\varrho,y}$ is Lipschitz 
with Lipschitz constant at most $1 + \Lip\left( \varrho \right)$. 
Moreover, if $\Lip\left( \varrho \right) < 1/2$,
then $\Phi_{\varrho,y}$ is bi-Lipschitz 
with an inverse whose Lipschitz constant is bounded by $2$.

For every $u \in L^p\Lambda^{k}(\Omega^{\rm e})$, $p \in [1,\infty]$, we then define 
\begin{align}
 \label{math:mollification_operator}
 R^{k}_{\varrho} u_{|x}
 &:=
 \int_{\bbR^n} \mu( y ) (\Phi_{\varrho,y}^{\ast} u)_{|x} \dif y,
 \quad
 x \in \overline\Omega
 .
\end{align}
The mapping $R^{k}_{\varrho}$ has range in $C^{\infty}\Lambda^{k}(\overline\Omega)$,
commutes with the exterior derivative, and satisfies local bounds in the supremum norm.

\begin{lemma}[Lemma~7.4 in \cite{licht2016smoothed}]
 \label{prop:mollifier}
 Assume that $\Phi_{\varrho,y} : \overline\Omega \rightarrow \Omega^{\rm e}$ is a LIP embedding
 for all $y \in B_1(0)$. 
 We have a well-defined linear operator 
 \begin{align*}
  R^{k}_{\varrho}
  :
  L^{p}\Lambda^{k}(\Omega^{\rm e}) \rightarrow C^{\infty}\Lambda^{k}(\overline\Omega),
  \quad 
  p \in [1,\infty]
  .
 \end{align*}
 For every $p \in [1,\infty]$,
 measurable $A \subseteq \overline\Omega$,
 and $u \in L^{p}\Lambda^{k}(\Omega^{\rm e})$
 we have 
 \begin{align}
  \label{math:mollifier_estimate_local}
  \| R^{k}_{\varrho} u \|_{C\Lambda^{k}(A)}
  &\leq
  \left( 1 + \Lip(\varrho) \right)^{k}
  \inf_{A}(\varrho)^{-\frac{n}{p}} 
  \| u \|_{L^p\Lambda^{k}\left( \Phi_{\varrho}(A,B_1) \right) }.
 \end{align}
 For every $u \in W^{p,q}\Lambda^{k}(\Omega^{\rm e})$ with $p,q \in [1,\infty]$
 we have $\cartan R^{k}_{\varrho} u = R^{k+1}_{\varrho} \cartan u$. 
 \qed
\end{lemma}

We are now in the position to combine the extension operator of Section~\ref{sec:geometry},
the distortion operator of Section~\ref{sec:distortion},
and regularization operator of Lemma~\ref{prop:mollifier}. 
Letting $\delta > 0$ be a small parameter to be determined below,
we define 
\begin{gather}
 \label{math:regularizer:definition}
 M^{k}_{\varrho} : L^{p}\Lambda^{k}(\Omega) \rightarrow C^{\infty}\Lambda^{k}(\Omega),
 \quad 
 u \mapsto R^{k}_{\delta\varrho} \frakD_{\varrho}^{\ast} E^{k} u,
 \quad 
 p \in [1,\infty].
\end{gather}
The properties of the mapping $M^{k}_{\varrho}$ are summarized as follows.

\begin{theorem}
 \label{prop:regularizer:main}
 Assume that $\varrho : \bbR^{n} \rightarrow \bbR$ satisfies the conditions of Lemma~\ref{prop:mollifier}
 and Theorem~\ref{prop:distortionmapping} applied to $\Upsilon$.
 Assume also that $\delta \in (0,1)$ with $2 \delta L_{D} < 1$. 
 Then $M^{k}_{\varrho}$ is well-defined. 
 Moreover, there exist $C^{M}_{n,k,p} > 0$ and $L_{M} > 0$,
 not depending on $\varrho$, 
 such that for all measurable $A \subseteq \overline\Omega$ and $u \in L^{p}\Lambda^{k}(\Omega)$ we have 
 \begin{align}
  \label{prop:regularizer:main:localbound}
  \| M^{k}_{\varrho} u \|_{C\Lambda^{k}(A)}
  \leq
  C^{M}_{n,k,p} 
  \dfrac{ 
   ( 1 + \Lip(\varrho) )^{k+\frac{n}{p}} 
  }{
   \inf_{A}(\varrho)^{\frac{n}{p}} 
  }
  \| u \|_{L^{p}\Lambda^{k}( B_{L_{M}(1+\Lip(\varrho)) \sup_{A}(\varrho) }(A) \cap \Omega )}
  . 
 \end{align}
 Additionally, if $p,q \in [1,\infty]$ and $u \in W^{p,q}\Lambda^{k}(\Omega,\Gamma_T)$
 then
 \begin{gather}
  \label{prop:regularizer:main:commutativity}
  M^{k+1}_{\varrho} \cartan u
  = 
  \cartan M^{k}_{\varrho} u,
 \end{gather}
 and $M^{k}_{\varrho} u$ vanishes in a neighborhood of $\Gamma_T$.
\end{theorem}

\begin{theconstants*}
 We may assume that $L_{M} \leq L_{D} L_{E}$ and 
  $C^{M}_{n,k,p}
  \leq  
  \delta^{-\frac{n}{p}} L_D^{k+\frac{n}{p}} \left( 1 + C_{E}^{ k + \frac{n}{p} } \right)$.
\end{theconstants*}

\begin{proof}
 We combine Theorem~\ref{prop:extensionoperator}, 
 Theorem~\ref{prop:distortionmapping} together with \eqref{math:pullbackestimate},
 and Theorem~\ref{prop:mollifier}
 to find that $M^{k}_{\varrho}$ as given by \eqref{math:regularizer:definition} is well-defined.
 Similarly we deduce \eqref{prop:regularizer:main:commutativity}. 
 Let us now assume
 that $u \in L^{p}\Lambda^{k}(\Omega)$ 
 and 
 that $A \subseteq \Omega$ is measurable. 
 
 We prove the local estimate \eqref{prop:regularizer:main:localbound}.
 Via Theorem~\ref{prop:mollifier} we find   
 \begin{gather}
  \label{math:anwendung:regularizer}
  \| R^{k}_{\delta\varrho} \frakD_{\varrho}^{\ast} E^{k} u \|_{C\Lambda^{k}(A)}
  \leq
  \frac
  { \left( 1 + \Lip(\varrho) \right)^{k} }
  { \left( \delta \inf_{A}(\varrho) \right)^{\frac{n}{p}} } 
  \| \frakD_{\varrho}^{\ast} E^{k} u \|
  _{L^{p}\Lambda^{k}\left( \Phi_{\delta\varrho}(A,B_1) \right)}
  .
 \end{gather}
 By Theorem~\ref{prop:distortionmapping} and the pullback estimate \eqref{math:pullbackestimate} we have 
 \begin{align*}
   &\| \frakD_{\varrho}^{\ast} E^{k} u \|
   _{L^{p}\Lambda^{k}\left( \Phi_{\delta\varrho}(A,B_1) \right)}
   \leq
   L_{D}^{k + \frac{n}{p}}
   ( 1 + \Lip(\varrho) )^{k+\frac{n}{p}}
   \| E^{k} u \|_{L^{p}\Lambda^{k}\left( \frakD_{\varrho}\Phi_{\delta\varrho}(A,B_1) \right)}
   .
 \end{align*}
 For $x \in A$ and $y \in B_{\delta\varrho(x)}(x)$ we use \eqref{prop:distortionmapping:shiftbound} to find 
 \begin{align*}
  \| \frakD_{\varrho}(y) - y \| 
  \leq 
  L_{D} \varrho(y) 
  &\leq 
  L_{D}\big( \varrho(x) + | \varrho(y) - \varrho(x) | \big)
  \\&\leq 
  L_{D}\big( \varrho(x) + \Lip(\varrho) \| y - x \|  \big)
  \leq 
  L_{D} \varrho(x) \big( 1 + \delta \Lip(\varrho) \big)
  .
 \end{align*}
 Consequently,
 the definition of $\Phi$ now gives 
 \begin{align*}
  \frakD_{\varrho}\Phi_{\delta\varrho}(A,B_1)
  &\subseteq
  B_{L_D\left( 1 + \delta \Lip(\varrho) \right) \sup_{A}(\varrho) }(A) \cap \Omega^{\rm e}
  .
 \end{align*}
 The desired inequality \eqref{prop:regularizer:main:localbound} follows with Lemma~\ref{prop:extensionoperator}.
 \begin{align*}
  \| E^{k} u \|
  _{L^{p}\Lambda^{k}\left( \frakD_{\varrho}\Phi_{\delta\varrho}(A,B_1) \right)}
  \leq 
  \left( 1 + C_{E}^{ k + \frac{n}{p} } \right)
  \| u \|
  _{L^{p}\Lambda^{k}\left( B_{L_{E}L_D\left( 1 + \delta \Lip(\varrho) \right) \sup_{A}(\varrho) }(A) \cap \Omega \right)}
  . 
 \end{align*}
 We finish the proof by showing that $M^{k}_{\varrho} u$ 
 vanishes near $\Gamma_T$.
 Let $x \in \Gamma_T$ and consider $A = B_{\Lip(\varrho)^{\inv} \varrho(x)}(x) \cap \overline\Omega$. 
 For all $y \in A$ we have 
 \begin{gather*}
  \varrho(y) \leq \varrho(x) + \Lip(\varrho) \| y - x \| \leq 2 \varrho(x). 
 \end{gather*}
 Hence $B_{ \delta \sup_{A}(\varrho) }$ 
 is a subset of 
 $B_{2\delta\varrho(x)}(x)$.
 In particular, \eqref{math:anwendung:regularizer} now shows that 
 the maximum of $R^{k}_{\delta\varrho} \frakD_{\varrho}^{\ast} E^{k} u$ over $A$
 is bounded by a multiple of the $L^{p}$ norm of $\frakD_{\varrho}^{\ast} E^{k} u$
 over $B_{2\delta\varrho(x)}(x)$.
 One more application of the pullback estimate \eqref{math:pullbackestimate} gives 
 \begin{gather*}
 \| \frakD_{\varrho}^{\ast} E^{k} u \|
  _{L^{p}\Lambda^{k}\left( B_{2\delta\varrho(x)}(x) \right)}
  \leq 
  L_{D}^{k + \frac{n}{p}}
  ( 1 + \Lip(\varrho) )^{k+\frac{n}{p}}
  \| E^{k} u \|
  _{L^{p}\Lambda^{k}\left( \frakD_{\varrho}B_{2\delta\varrho(x)}(x) \right)}
  . 
 \end{gather*}
 We now assume $2\delta < 1 / L_D$.
 In combination with Theorem~\ref{prop:distortionmapping} we conclude 
 \begin{gather*}
  \frakD_{\varrho}B_{2\delta\varrho(x)}(x) \subseteq \Upsilon
  .
 \end{gather*}
 But $E^{k} u$ is zero over $\Upsilon$.
 The proof is complete. 
\end{proof}

Although the focus of this research is numerical analysis,
the results up to this point allow a contribution to functional analysis.
Specifically, we prove the density result mentioned in the introduction.

\begin{lemma}
 Let $\Omega$ be a bounded weakly Lipschitz domain  and let $\Gamma_T$ be an admissible boundary patch.
 Then the smooth differential $k$-forms in $C^{\infty}\Lambda^{k}(\overline\Omega)$
 that vanish near $\Gamma_T$
 constitute a dense subset of $W^{p,q}\Lambda^{k}(\Omega,\Gamma_T)$
 for all $p,q \in [1,\infty)$.
\end{lemma}

\begin{proof}
 Let $p,q \in [1,\infty)$ and $u \in W^{p,q}\Lambda^{k}(\Omega,\Gamma_T)$.
 For $\eps > 0$ small enough we define $M^{k}_{\eps} u$ using Theorem~\ref{prop:regularizer:main}.
 We also set ${\Omega'_\eps} := \Omega \cap B_{L_D \eps}\left( \partial\Omega \right)$
 and ${\Omega''_\eps} := \Omega \setminus {\Omega'_\eps}$.
 On the latter set, $M^{k}_{\eps} u$ agrees with the classical convolution of $u$ with $\mu_{\delta\eps}$.
 Hence 
 \begin{align*}
  \left\| u - M^{k}_{\eps} u \right\|_{L^{p}\Lambda^{k}({\Omega''_\eps})}
  &=
  \left\| u - \mu_{\delta\eps} \star u \right\|_{L^{p}\Lambda^{k}({\Omega''_\eps})}
  \leq 
  \left\| u - \mu_{\delta\eps} \star u \right\|_{L^{p}\Lambda^{k}(\Omega)}
  . 
 \end{align*}
 By basic results on mollifications,
 the last expression converges to zero as $\eps$ converges to zero.
 Going further, a combination of Young's inequality for convolutions, 
 the pullback estimate \eqref{math:pullbackestimate} applied to $\frakD_{\eps}$,
 and Theorem \eqref{prop:extensionoperator} show the existence of $C > 0$
 independent of $u$ and $\eps$ such that 
 \begin{gather*}
  \left\| u - M^{k}_{\eps} u \right\|_{L^{p}\Lambda^{k}({\Omega'_\eps})}
  \leq 
  C \left\| u \right\|_{L^{p}\Lambda^{k}(\Omega \cap {\Omega'_{C\eps}} )}
  . 
 \end{gather*}
 Using that $\Omega$ is a weakly Lipschitz domain, 
 one can show that $\vol^{n}({\Omega'_{C\eps}})$ converges to zero
 as $\eps$ converges to zero.
 This implies that the $L^{p}$ norm of $u$ over ${\Omega'_\eps} \cap \Omega$ 
 converges to zero as $\eps$ converges to zero.
 Since $M_{\eps}^{k+1} \cartan u = \cartan M_{\eps}^{k} u$,
 we conclude that $M^{k}_{\eps} u$ converges to $u$ in $W^{p,q}\Lambda^{k}(\Omega)$. 
 Lastly, we recall that $M^{k}_{\eps} u \in C^{\infty}\Lambda^{k}(\overline\Omega)$
 with support away from $\Gamma_T$ for all $\eps > 0$.
 The proof is complete.
\end{proof}

%
%

\section{Finite Element Projection with Partial Boundary Conditions}
\label{sec:projection}

In this section, we prove the main result of this paper: 
we construct uniformly bounded commuting projections from Sobolev de~Rham complexes with partial boundary conditions
onto conforming finite element de~Rham complexes. 
Throughout this section, we let $\Omega \subseteq \bbR^{n}$ be a fixed weakly Lipschitz domain 
and we let $\Gamma_T \subseteq \partial\Omega$ be an admissible boundary patch.
\\

The discussion of finite element spaces requires a brief discussion of triangulations.
We follow \cite[Section~4]{licht2016smoothed}, to which the reader is referred for details.
We fix a finite simplicial complex $\calT$ that triangulates $\overline\Omega$. 
We write $\calT^{m}$ for the set of $m$-dimensional simplices in $\calT$,
and for each $T \in \calT$ we write 
\begin{align*}
 \Delta(T) := \left\{\; S \in \calT \suchthat S \subseteq T \;\right\}
 ,
 \quad
 \calT(T) := \left\{\; S \in \calT \suchthat S \cap T \neq \emptyset \;\right\}
 .
\end{align*}
We work with finite element spaces of differential forms 
that constitute a finite element de~Rham complex in the sense 
of finite element exterior calculus \cite{AFW1,AFW2}. 
We let $\calP\Lambda^{k}(\calT)$ be one of the spaces 
$\calP_{r  }\Lambda^{k  }(\calT)$ and $\calP_{r  }^{-}\Lambda^{k  }(\calT)$, 
as defined in \cite[Section~3--4]{AFW1}, 
such that the choice of finite element spaces yields a finite element de~Rham complex 
\begin{align}
 \label{math:finiteelementcomplex}
 \begin{CD}
  0 \to
  \calP\Lambda^{0}(\calT)
  @>\cartan>>
  \calP\Lambda^{1}(\calT)
  @>\cartan>>
  \cdots
  @>\cartan>>
  \calP\Lambda^{n}(\calT)
  \to 0
 \end{CD}
\end{align}
of Arnold-Falk-Winther-type as described in \cite[Section~3--4]{AFW1}. 
This class of finite element de Rham complexes admit commuting canonical interpolants
\begin{align}
 \label{math:FEinterpolant}
 I^{k} : 
 C\Lambda^{k}(\overline\Omega)
 + W^{\infty,\infty}\Lambda^{k}(\Omega)
 \rightarrow 
 \calP\Lambda^{k}(\calT)
\end{align}
that are bounded operators on continuous and flat differential forms,
that act as the identity on the finite element spaces, 
and that commute with the exterior derivative on flat differential forms. 

The definition of finite element spaces with boundary conditions 
requires an additional geometric assumption. 
We assume from now on that there exists a subset $\calU \subset \calT$
whose member's union is exactly $\overline\Gamma_{T}$. 
In other words, $\calU$ triangulates the closure of the boundary patch. 
Setting 
\begin{align}
 \calP\Lambda^{k}(\calT,\calU)
 &:=
 \left\{\; 
  u \in \calP\Lambda^{k}(\calT)
  \suchthat 
  \forall F \in \calU : \trace_{F} u = 0
 \;\right\}
 ,
\end{align}
we observe that 
$\calP\Lambda^{k}(\calT,\calU) = \calP\Lambda^{k}(\calT) \cap W^{\infty,\infty}(\Omega,\Gamma_{T})$.
The discussion in \cite[Section~6]{licht2016smoothed} 
easily gives the well-definedness of the mapping 
\begin{gather*}
 I^{k} : 
 W^{\infty,\infty}\Lambda^{k}(\Omega,\Gamma_T) 
 \rightarrow 
 \calP\Lambda^{k}(\calT,\calU)
 .
\end{gather*}
This means that the canonical interpolant preserves homogeneous boundary conditions. 
In particular, the following diagram commutes:
\begin{align}
 \label{math:interpolationoperatorcommutingdiagram}
 \begin{CD}
  \cdots
  @>>>
  W^{\infty,\infty}\Lambda^{k  }(\Omega,\Gamma_T)
  @>{\cartan}>>
  W^{\infty,\infty}\Lambda^{k+1}(\Omega,\Gamma_T)
  @>>>
  \cdots
  \\
  @.
  @V{I^{k}}VV
  @VI^{k+1}VV
  @.
  \\
  \cdots
  @>>>
  \calP\Lambda^{k  }(\calT,\calU) 
  @>{\cartan}>> 
  \calP\Lambda^{k+1}(\calT,\calU) 
  @>>> 
  \cdots
 \end{CD}
\end{align}
We review some technical definitions and inequalities in the proofs of our main results
(see, again, \cite[Section~4]{licht2016smoothed}).
For every $m$-dimensional simplex $T \in \calT^{m}$ 
we let $h_T = \diam(T)$ denote its diameter 
and $\vol^{m}(T)$ denote its $m$-dimensional volume. 
If $V \in \calT^{0}$, then convention 
has $\vol^{0}(V) = 1$ and $h_V$ be the average length of all $n$-simplices of $\calT$
adjacent to $V$. 
The \emph{shape constant} of $\calT$ is the minimal $C_{\rm mesh} > 0$ 
such that for all $T \in \calT^{n}$ we have $h_T^n \leq C_{\rm mesh} \vol^{n}(T)$
and for all $S, T \in \calT$ with $S \cap T \neq \emptyset$
we have $h_T \leq C_{\rm mesh} h_S$.
Intuitively, this constant measures the regularity of the triangulation. 
In the sequel, we call a constant \emph{uniformly bounded}
if it can be bounded in terms of geometric properties of the domain, 
the shape constant, and the maximal polynomial degree 
of the finite element spaces. 

There exist uniformly bounded constants $C_{\rm N}, \eps_{h} > 0$
such that 
\begin{gather}
 \forall T \in \calT 
 :
 |\calT(T)| \leq C_{\rm N}
 , 
 \\
 \label{math:neighborcontainment:duo}
 \forall T \in \calT 
 :
 B_{\eps_{h} h_T}(T) \cap \overline\Omega \subseteq \calT(T).
\end{gather}
In addition, we recall Lemma~7.7 in \cite{licht2016smoothed}, 
which asserts the existence of a smooth function 
that indicates the local mesh size over the domain. 
More precisely, there exist uniformly bounded constants 
$L_h > 0$ and $C_h > 0$ and a smooth function 
$\mathtt h : \bbR^{n} \rightarrow \bbR$ 
such that 
\begin{gather}
 \label{math:meshfunction:lipschitz}
 \Lip( \mathtt h, \overline\Omega ) 
 \leq
 L_{h},
 \quad 
 \min_{T \in \calT^{n} } h_{T}
 \leq 
 \min_{ x \in \bbR^{n} } \mathtt h(x)
 ,
 \quad 
 \max_{ x \in \bbR^{n} } \mathtt h(x)
 \leq 
 \max_{T \in \calT^{n} } h_{T}
 ,
 \\
 \label{math:meshfunction:localcomparison}
 \forall F \in \calT : \forall x \in F : 
 C_{h}^{\inv} h_F \leq \mathtt h(x) \leq C_{h} h_F
 .
\end{gather}
For the discussion of scaling arguments, 
we recall the definition of the $n$-dimensional reference simplex
$\Delta^{n} = \convex\{0,e_1,\dots,e_n\} \subseteq \bbR^{n}$.
For each $T \in \calT^{n}$ we fix an affine transformation
$\varphi_T(x) = M_T x + b_T$
with $b_T \in \bbR^{n}$ and invertible $M_T \in \bbR^{n \times n}$ 
such that $\varphi_T( \Delta^{n} ) = T$.
We can estimate the $\ell^{2}$ operator norms 
\begin{align}
 \label{math:referencetransformation}
 \left\| M_T \right\|
 \leq c_{\rm M} h_T
 , 
 \quad 
 \left\| M_T^{\inv} \right\|
 \leq C_{\rm M} h_T^{\inv}
 .
\end{align}
with uniformly bounded constants $c_{\rm M}, C_{\rm M} > 0$.
\\

Having reviewed these technical definitions, 
we discuss the main results of this section.  
For $\eps > 0$ small enough, 
we define the \emph{smoothed interpolant} $Q^{k}_{\eps}$ as the operator 
\begin{align}
 Q^{k}_{\eps} 
 :
 L^{p}\Lambda^{k}(\Omega)
 \rightarrow
 \calP\Lambda^{k}(\calT,\calU)
 \subseteq 
 L^{p}\Lambda^{k}(\Omega), 
 \quad 
 u
 \mapsto
 I^{k} M_{\eps {\mathtt h}}^{k} u,
 \quad 
 p \in [1,\infty]. 
\end{align}

\begin{theorem}
 \label{prop:interpolationbound}
 Let $\eps > 0$ be small enough.
 For $p \in [1,\infty]$
 there exists a uniformly bounded constant $C_{Q,p} > 0$ such that 
 \begin{gather}
  \label{prop:interpolationbound:localestimate}
  \| Q^{k}_{\eps} u \|_{L^{p}\Lambda^{k}(T)}
  \leq
  C_{Q,p} \eps^{ -\frac{n}{p} } \| u \|_{L^{p}\Lambda^{k}(\calT(T))},
  \quad 
  u \in L^{p}\Lambda^{k}(\Omega),
  \quad T \in \calT^{n}, 
  \\
  \label{prop:interpolationbound:globalestimate}
  \| Q^{k}_{\eps} u \|_{L^{p}\Lambda^{k}(\Omega)}
  \leq
  C_{\rm N}^{1/p} C_{Q,p} \eps^{ -\frac{n}{p} } \| u \|_{L^{p}\Lambda^{k}(\Omega)},
  \quad 
  u \in L^{p}\Lambda^{k}(\Omega)
  .
 \end{gather}
 Moreover, we have 
 \begin{align}
  \label{prop:interpolationbound:commutativity}
  \cartan Q^{k}_{\eps} u = Q^{k}_{\eps} \cartan u,
  \quad 
  u \in W^{p,q}\Lambda^{k}(\Omega,\Gamma_T),
  \quad 
  p, q \in [1,\infty].
 \end{align}
\end{theorem}

\begin{theconstants*}
 It suffices that $\eps > 0$ is so small that $L_{\rm M}(1+\eps L_{h}) C_{h} \eps < \eps_{h}$
 and that Theorem~\ref{prop:regularizer:main} applies for $\varrho = \eps \mathtt{h}$.
 With the notation as in the following proof, we may assume 
 $C_{Q,p} \leq C_{\rm M}^{k} c_{\rm M}^{k} C_{I} C_{h} C^{M}_{n,k,p} \left( 1 + \eps L_{h} \right)^{\frac{n}{p}}$.
\end{theconstants*}

\begin{proof}
 The proof is very similar to the proof of Theorem~7.8 in \cite{licht2016smoothed}. 
 We let $u \in L^{p}\Lambda^{k}(\Omega)$ and $T \in \calT^{n}$.
 By \eqref{math:pullbackestimate}, \eqref{math:referencetransformation} and $\vol^{n}(T) \leq h_T^{n}$ we get 
 \begin{align*}
  \| Q^{k}_{\eps} u \|_{L^{p}\Lambda^{k}(T)}
  &\leq 
  C_{\rm M}^{k} h_T^{\frac{n}{p}-k}
  \| \varphi_T^{\ast} I^{k} M^{k}_{\eps {\mathtt h}} u \|_{L^{\infty}\Lambda^{k}(\Delta^{n})}
  . 
 \end{align*}
 By estimate (6.15) of \cite{licht2016smoothed} and discussions in that reference 
 we know about the existence of a uniformly bounded constant $C_{I} > 0$ such that 
 \begin{align*}
  \| \varphi_T^{\ast} I^{k} M^{k}_{\eps {\mathtt h}} u \|_{L^{\infty}\Lambda^{k}(\Delta^{n})}
  &\leq 
  C_{I} c_{\rm M}^{k} h_T^{k}
  \| M^{k}_{\eps {\mathtt h}} u \|_{C\Lambda^{k}(T)}.
 \end{align*}
 Assuming that $\eps$ is small enough, we can apply Theorem~\ref{prop:regularizer:main} to find 
 \begin{align*}
  \| M^{k}_{\eps {\mathtt h}} u \|_{C\Lambda^{k}(T)}
  \leq
  C^{M}_{n,k,p}
  \dfrac{ 
   C_{h} ( 1 + \eps L_{h} )^{k+\frac{n}{p}} 
  }{
   \left( \eps h_{T} \right)^{\frac{n}{p}}
  }
  \| u \|_{L^{p}\Lambda^{k}\left( B_{L_{\rm M}(1+\eps L_{h}) C_{h} \eps h_T }(T) \cap \Omega \right)}
  . 
 \end{align*}
 If $L_{\rm M}(1+\eps L_{h}) C_{h} \eps < \eps_{h}$, then $B_{ L_{\rm M}(1+\eps L_{h}) C_{h} \eps h_T }(T) \subseteq \calT(T)$.
 Thus the local bound \eqref{prop:interpolationbound:localestimate} is proven,
 and the global bound \eqref{prop:interpolationbound:globalestimate} follows easily with 
 \begin{gather*}
  \sum_{ T \in \calT^{n} }
  \| \omega \|_{L^{p}\Lambda^{k}( T )}^{p}
  \leq
  \sum_{ T \in \calT^{n} }
  \| \omega \|_{L^{p}\Lambda^{k}( \calT(T) )}^{p}
  \leq C_{\rm N}
  \sum_{ T \in \calT^{n} }
  \| \omega \|_{L^{p}\Lambda^{k}( T )}^{p}
  .
 \end{gather*}
 Moreover, 
 $M^{k}_{\eps \mathtt h} E^{k} u$ vanishes near every $F \in \calU$. 
 Finally, \eqref{prop:interpolationbound:commutativity}
 follows from Theorem~\ref{prop:extensionoperator},
 Theorem~\ref{prop:mollifier}, and our assumptions on $I^{k}$. 
 The proof is complete. 
\end{proof}


We have proven uniform local bounds for the smoothed interpolant $Q_{\eps}^{k}$.
Even though $Q_{\eps}^{k}$ is generally not a projection onto $\calP\Lambda^{k}(\calT,\calU)$,
the interpolation error over $\calP\Lambda^{k}(\calT,\calU)$
can be made arbitrarily small when $\eps > 0$ is small enough.

\begin{theorem}
 \label{prop:interpolation_error}
 For $\eps > 0$ satisfying a uniform bound, 
 there exists uniformly bounded $C_{e,p} > 0$ for every $p \in [1,\infty]$ 
 such that 
 \begin{align*}
  \| u - Q_{\eps}^{k} u \|_{L^{p}\Lambda^{k}(T)}
  \leq 
  \eps C_{e,p} 
  \| u \|_{L^{p}\Lambda^{k}(\calT(T))},
  \quad 
  u \in \calP\Lambda^{k}_{}(\calT),
  \quad 
  T \in \calT^{n}
  .
 \end{align*}
\end{theorem}

\begin{theconstants*}
 With the notation as in the following proof, 
 it suffices that $\eps > 0$ is so small that Theorem~\ref{prop:interpolationbound} applies,  
 that $L_E \calL \eps < \eps_{h}$, and that $\calL \eps < 1/3$.   
 In addition, 
 we may assume
 $C_{e,p}
  \leq 
  C^{2k+1+\frac{n}{p}}_{\rm M} c_{\rm M}^{2k+1}
  C_{I}
  \left( 1 + C_{E}^{k+1+\frac{n}{p}} \right)
  C_{\flat,p}
  {\calL} 
  \max(1,\frakL)^{k} C_{\partial}
  $.
\end{theconstants*}

\begin{proof} 
 The proof here is a technical modification of the proof of Theorem~7.9 in \cite{licht2016smoothed},
 to which the reader is referred at all times for further details.  
 Let $u \in \calP\Lambda_{}^{k}(\calT)$ and let $T \in \calT^{n}$. 
 Using \eqref{math:pullbackestimate}, \eqref{math:referencetransformation}, and definitions, 
 we verify  
 \begin{align*}
  \| u - Q_{\eps}^{k} u \|_{L^{p}\Lambda^{k}(T)}
  &\leq 
  h_T^{\frac{n}{p}}
  \| E^{k} u - Q_{\eps}^{k} u \|_{L^{\infty}\Lambda^{k}(T)}
  \\&\leq 
  C_{\rm M}^{k} h_T^{\frac{n}{p}-k}
  \| \varphi_{T}^{\ast} I^{k} ( E^{k} u - R^{k}_{\delta \eps \mathtt h} \frakD_{\eps \mathtt h}^{\ast} E^{k} u ) \|_{L^{\infty}\Lambda^{k}(\Delta^{n})}
  .
 \end{align*}
 To proceed with the proof, 
 we need to recall the definition of the canonical interpolant 
 via degrees of freedom as in Section~6 of \cite{licht2016smoothed}. 
 We let $\calP\calC_{k}(F)$ denote the space of degrees of freedom associated 
 with the subsimplex $F \in \Delta(T)$;
 these spaces of functionals are given by taking the trace of a differential $k$-form
 onto the subsimplex $F$ and then integrating against another polynomial differential
 form over $F$ (see also Remark~6.1 of \cite{licht2016smoothed}).
 
 There exists a uniformly bounded constant $C_{I} > 0$,  
 as in Inequality (6.14) of \cite{licht2016smoothed}, 
 for which one can show 
 \begin{align*}
  &
  \| \varphi_{T}^{\ast} I^{k} ( E^{k} u - R^{k}_{\delta \eps \mathtt h} \frakD_{\eps \mathtt h}^{\ast} E^{k} u ) \|
  _{L^{\infty}\Lambda^{k}(\Delta^{n})}
  \\&\qquad 
  \leq 
  C_I
  \sup_{ \substack{ F \in \Delta(T) \\ S \in \calP\calC_{k}^{F} } }
  \left| \varphi_{T\ast}^{\inv}S \right|^{\inv}_{k}
   \int_{S} 
   E^{k} u
   -
   R^{k}_{\delta \eps \mathtt h} \frakD_{\eps \mathtt h}^{\ast} E^{k} u
   .
 \end{align*}
 Here, $|\varphi_{T\ast}^{\inv} S|_{k}$ is defined in the following manner.
 Let $S \in \calP\calC^{F}_{k}$ is given as the integral over $F \in \Delta(T)^{m}$
 against the smooth differential form $\eta_{S} \in C^{\infty}\Lambda^{m-k}(F)$.
 Let $\widehat F \in \Delta^{n}$ be the unique $m$-simplex that $\varphi_{T}$ maps onto $F$. 
 Then $|\varphi_{T\ast}^{\inv} S|_{k}$ equals the $L^{1}$ norm 
 of $\varphi_{T}^{\ast} \eta_{S}$ over $\widehat F$. 
 This is equivalent to the definition of $|\varphi_{T\ast}^{\inv} S|_{k}$
 via the mass norm of $k$-chains as used in \cite{licht2016smoothed}. 
 
 Fix $F \in \Delta(T)$ and $S \in \calP\calC^{F}_{k}$.
 We have 
 \begin{align*}
  \int_{S}
   E^{k} u
   -
   R^{k}_{\delta \eps \mathtt h} \frakD_{\eps \mathtt h}^{\ast} E^{k} u
  =
  \int_{S} 
  \int_{\bbR^{n}}
  \mu(y) 
  \left( 
   \Id 
   -
   \Phi_{\eps {\mathtt h},y}^{\ast} \frakD_{\eps \mathtt h}^{\ast} 
  \right) 
  E^{k} u
  \dif y.
 \end{align*}
 We then change the order of integration: 
 \begin{gather*}
  \int_{S} 
  \int_{\bbR^{n}}
  \mu(y) 
  \left( 
   \Id 
   -
   \Phi_{\eps {\mathtt h},y}^{\ast} \frakD_{\eps \mathtt h}^{\ast} 
  \right) 
  E^{k} u
  \dif y
  =
  \int_{\bbR^{n}}
  \mu(y) 
  \int_{S} 
  \left( 
   \Id 
   -
   \Phi_{\eps {\mathtt h},y}^{\ast} \frakD_{\eps \mathtt h}^{\ast} 
  \right) 
  E^{k} u
  \dif y
  .
 \end{gather*}
 Using Equation (5.14) of \cite{licht2016smoothed}, we see for $y \in B_{1}(0)$ that 
 \begin{align}
 \label{math:zwischenschritt}
  \int_{ S} 
  \left( 
   \Id 
   -
   \Phi_{\eps {\mathtt h},y}^{\ast} \frakD_{\eps \mathtt h}^{\ast} 
  \right) 
  E^{k} u
  =
  \int_{ \varphi_{T\ast}^{\inv} \left( \Id - \frakD_{\eps \mathtt h\ast} \Phi_{\delta \eps {\mathtt h},y \ast} \right) S } 
   \varphi_{T}^{\ast} E^{k} u
  .
 \end{align}
 In the remainder of the proof we bound the last term.
 We need an auxiliary estimate that bounds the difference 
 $\Id - \varphi_{T}^{\inv} \Phi_{\delta \eps {\mathtt h},y} \frakD_{\eps \mathtt h} \varphi_{T}$ uniformly in terms of $\eps$ and $y$
 within a small radius of $\varphi_{T}^{\inv} F$.
 First we see 
 \begin{gather*}
  \sup_{ y \in B_1(0) }
  \Lip\left( \varphi_{T}^{-1} \Phi_{\delta\eps\mathtt h,y} \frakD_{\eps \mathtt h} \varphi_{T} \right)
  \leq 
  c_{\rm M} C_{\rm M} L_{D} \left( 1 + \eps L_h \right)^{2}
  =:
  \frakL.
 \end{gather*}
 For any $y \in B_{1}(0)$ and $\hat x \in B_{1}( \varphi_{T}^{\inv} F )$ we find that 
 \begin{align*}
  &
  \vectornorm{ \hat x - \varphi_T^{\inv} \Phi_{\delta\eps\mathtt h,y} \frakD_{\eps \mathtt h} \varphi_T (\hat x) } 
  \\&\quad
  \leq 
  C_{\rm M} h_{T}^{\inv}
  \vectornorm{ \varphi_{T}(\hat x) - \Phi_{\delta\eps\mathtt h,y} \frakD_{\eps \mathtt h} \varphi_T (\hat x) } 
  \\&\quad
  \leq 
  C_{\rm M} h_{T}^{\inv}
  \vectornorm{ \varphi_{T}(\hat x) - \frakD_{\eps \mathtt h} \varphi_T (\hat x) } 
  +
  C_{\rm M} h_{T}^{\inv}
  \vectornorm{ \frakD_{\eps \mathtt h} \varphi_T (\hat x) - \Phi_{\delta\eps\mathtt h,y} \frakD_{\eps \mathtt h} \varphi_T (\hat x) } 
  . 
 \end{align*}
 We then estimate 
 \begin{gather*}
  \vectornorm{ \varphi_{T}(\hat x) - \frakD_{\eps \mathtt h} \varphi_T (\hat x) } 
  \leq
  L_{D} \eps \mathtt h\left( \varphi_{T}(\hat x) \right)
 \end{gather*}
 and 
 \begin{align*}
  \vectornorm{ \frakD_{\eps \mathtt h} \varphi_T (\hat x) 
  - 
  \Phi_{\delta\eps\mathtt h,y} \frakD_{\eps \mathtt h} \varphi_T (\hat x) } 
  &\leq 
  \delta \eps \mathtt h\left( \frakD_{\eps \mathtt h} \varphi_T (\hat x) \right)
  \\&\leq 
  \delta \eps \mathtt h\left( \varphi_T (\hat x) \right) + \delta \eps L_{h} \vectornorm{ \varphi_T (\hat x) - \frakD_{\eps \mathtt h} \varphi_T (\hat x) }
  \\&\leq 
  \delta \eps \left( 1 + \delta \eps L_{h} L_{D} \right) \mathtt h\left( \varphi_T (\hat x) \right)
  . 
 \end{align*}
 Let $x_{F} \in F$ such that $\| \hat x - \varphi_{T}^{\inv}( x_{F} ) \| \leq 1$. 
 Then $\| \varphi_{T}( \hat x ) - x_{F} \| \leq c_{\rm M} h_{T}$. 
 Hence 
 \begin{gather*}
  \mathtt h( \varphi_{T}(\hat x) )
  \leq 
  \mathtt h( x_{F} ) + L_{h} c_{\rm M} h_{T} 
  \leq 
  \left( C_{h} + L_{h} c_{\rm M} \right) h_{T}
  .
 \end{gather*}
 Writing ${\calL} := C_{\rm M} \left( 1 + L_{D} + L_{h} L_{D} \right) \left( C_{h} + L_{h} c_{\rm M} \right)$
 and assuming $\eps \leq 1$ for simplicity,  
 we get 
 \begin{gather*}
  \sup_{ \substack{ \hat x \in B_{1}( \varphi_{T}^{\inv} F ) } }
  \sup_{ y \in B_{1}(0) }
  \vectornorm{ \hat x - \varphi_T^{\inv} \Phi_{\delta\eps\mathtt h,y} \frakD_{\eps \mathtt h} \varphi_T (\hat x) } 
  \leq 
  \eps {\calL}
  .
 \end{gather*}
 We continue with the main part of the proof. 
 Let $\eps > 0$ be so small that ${\calL} \eps < 1 / 3$.
 We apply Lemma~5.4 in \cite{licht2016smoothed} with $r = 1 / 3$
 and the inverse inequality~(6.13) in the same reference:  
 there exists uniformly bounded $C_{\partial} > 0$
 such that for all $y \in B_{1}(0)$ we have 
 \begin{align*}
  &
  \int_{ 
   \varphi_{T\ast}^{\inv} S
   - 
   \varphi_{T\ast}^{\inv} \frakD_{\eps \mathtt h \ast} \Phi_{\eps {\mathtt h},y \ast} S
  } 
  \varphi_{T}^{\ast}
  E^{k} u 
  \\&\quad\quad
  \leq
  \eps \cdot {\calL} \max(1,\frakL)^{k} C_{\partial}
  \cdot 
  |\varphi_{T\ast}^{\inv} S|_{k}
  \cdot 
  \| \varphi_{T}^{\ast} E^{k} u \|
  _{W^{\infty,\infty}\Lambda^{k}(B_{{\calL}\eps}(\Delta_{n}))},  
 \end{align*}
 As in the proof of Theorem~7.9 of \cite{licht2016smoothed}, 
 we observe 
 \begin{align*}
  \| \varphi_{T}^{\ast} E^{k} u \|_{W^{\infty,\infty}\Lambda^{k}(B_{{\calL}\eps}(\Delta_{n}))}  
  \leq
  \left( 1 + C_{E}^{ k + 1 + \frac{n}{p} } \right)
   c_{\rm M}^{k+1} C_{\rm M}^{k+1}
  \| \varphi_{T}^{\ast} u \|_{W^{\infty,\infty}\Lambda^{k}(\varphi_{T}^{\inv}\calT(T))}  
 \end{align*}
 for $\eps$ so small that $L_{E} {\calL} c_{\rm M} \eps < \eps_{h}$. 
 Now the inverse inequality~6.12 of \cite{licht2016smoothed} gives 
 \begin{align*}
  \| \varphi_{T}^{\ast} u \|_{W^{\infty,\infty}\Lambda^{k}(\varphi_{T}^{\inv}\calT(T))}
  \leq
  C_{\flat,p}
  \| \varphi_{T}^{\ast} u \|_{L^{p}\Lambda^{k}(\varphi_{T}^{\inv}\calT(T))}
  .
 \end{align*}
 Transforming back from the reference geometry yields 
 \begin{align*}
  \| \varphi_{T}^{\ast} u \|_{L^{p}\Lambda^{k}(\varphi_{T}^{\inv}\calT(T))}
  \leq 
  c_{\rm M}^{k} C_{\rm M}^{\frac{n}{p}}
  h_T^{k-\frac{n}{p}} 
  \| u \|_{L^{p}\Lambda^{k}(\calT(T))}
  .
 \end{align*}
 The combination of these inequalities completes the proof.
\end{proof}

For $\eps > 0$ small enough, the mapping 
$Q^{k}_{\eps} : \calP\Lambda^{k}(\calT,\calU) \rightarrow \calP\Lambda^{k}(\calT,\calU)$
is so close to the identity that it is invertible. 
We can then construct the smoothed projection. 

\begin{theorem}
 \label{prop:finalprojection}
 Let $\eps > 0$ be small enough.
 There exists a bounded linear operator 
 \begin{align*}
  \pi^{k}_{\eps} : L^{p}\Lambda^{k}(\Omega) \rightarrow \calP\Lambda^{k}(\calT,\calU) \subseteq L^{p}\Lambda^{k}(\Omega) ,
  \quad 
  p \in [1,\infty],
 \end{align*}
 such that 
 \begin{align*}
  \pi^{k}_{\eps} u = u, \quad u \in \calP\Lambda^{k}(\calT,\calU),
 \end{align*}
 such that  
 \begin{align*}
  \cartan \pi^{k}_{\eps} u = \pi^{k+1}_{\eps} \cartan u,
  \quad 
  u \in W^{p,q}\Lambda^{k}(\Omega,\Gamma_T),
  \quad
  p,q \in [1,\infty],
 \end{align*}
 and such that for all $p \in [1,\infty]$ there exist uniformly bounded $C_{\pi,p} > 0$ with
 \begin{align*}
  \| \pi^{k}_{\eps} u \|_{L^{p}\Lambda^{k}(\calT)}
  \leq 
  C_{\pi,p} \eps^{ -\frac{n}{p} } \| u \|_{L^{p}\Lambda^{k}(\Omega)},
  \quad 
  u \in L^{p}\Lambda^{k}(\Omega).
 \end{align*} 
\end{theorem}

\begin{theconstants*}
 It suffices that $\eps > 0$ is so small 
 that Theorem~\ref{prop:interpolationbound} and Theorem~\ref{prop:interpolation_error} apply,
 and that $C_{e,p} \eps < 2$.
 We may assume $C_{\pi,p} \leq 2 C_{Q,p} C_{\rm N}^{1/p}$.
\end{theconstants*}

\begin{proof}
 This is almost identically the proof 
 of Theorem~7.11 of \cite{licht2016smoothed}.
\end{proof}


%
%

\section{Applications}
\label{sec:application}

We conclude this article with an outline of the theoretical and numerical analysis 
of the Hodge Laplace equation with mixed boundary conditions. 
The smoothed projection is important for proving stability and convergence of a mixed finite element method 
based on a saddle point formulation of the Hodge Laplace equation.
This is similar to the theory of Hodge Laplace equation with non-mixed boundary conditions,
but the analytical background has only recently become available in the literature.
Moreover, the harmonic forms with mixed boundary conditions display some interesting qualities.

\subsection{Hodge Laplacian with Mixed Boundary Conditions}
\label{subsec:mixedhodgeplaplacian}

Throughout this section we assume that $\Omega$ is a bounded weakly Lipschitz domain
and that the tuple $(\Gamma_T,\Gamma_I,\Gamma_N)$ is an admissible boundary partition.
We write 
\begin{gather}
 H_T\Lambda^{k}(\Omega)
 :=
 W^{2,2}\Lambda^{k}(\Omega,\Gamma_T),
 \quad 
 H^\star_N\Lambda^{k}(\Omega,\Gamma_N)
 :=
 \star W^{2,2}\Lambda^{n-k}(\Omega,\Gamma_N)
 .
\end{gather}
These spaces are naturally Hilbert spaces. 
We know by Proposition 4.4 and Proposition 4.3(i)
of \cite{GMM} that the unbounded linear operators 
\begin{subequations}
\label{math:mixedbcdiff}
\begin{gather}
 \label{math:mixedbcdiff:tan}
 \cartan :
 H_T\Lambda^{k}(\Omega) \subseteq L^{2}\Lambda^{k}(\Omega)
 \rightarrow
 H_T\Lambda^{k+1}(\Omega)
 ,
 \\
 \label{math:mixedbcdiff:nor}
 \delta :
 H^\star_N\Lambda^{k}(\Omega) \subseteq L^{2}\Lambda^{k}(\Omega)
 \rightarrow
 H^\star_N\Lambda^{k-1}(\Omega)
\end{gather}
\end{subequations}
are densely-defined, closed, and mutually adjoint with closed range.
Thus 
\begin{subequations}
\label{math:derhamcomplexmixedbc}
\begin{align}
 \label{math:derhamcomplexmixedbc:tangential}
 \begin{CD}
  0 
  @>>>
  H_T\Lambda^{0}(\Omega)
  @>\cartan>>
  \dots
  @>\cartan>>
  H_T\Lambda^{n}(\Omega)
  @>>>
  0,
 \end{CD}
 \\
 \label{math:derhamcomplexmixedbc:normal}
 \begin{CD}
  0 
  @<<<
  H^\star_N\Lambda^{0}(\Omega)
  @<\delta<<
  \dots
  @<\delta<<
  H^\star_N\Lambda^{n}(\Omega)
  @<<<
  0
 \end{CD}
\end{align}
\end{subequations}
are closed Hilbert complexes in the sense of \cite{bruening1992hilbert}.
Moreover, \eqref{math:derhamcomplexmixedbc:tangential} and 
\eqref{math:derhamcomplexmixedbc:normal} are mutually adjoint.
We call \eqref{math:derhamcomplexmixedbc:tangential} the
\emph{$L^{2}$ de~Rham complex with tangential boundary conditions} along $\Gamma_T$,
and we call \eqref{math:derhamcomplexmixedbc:normal} the 
\emph{$L^{2}$ de~Rham complex with normal boundary conditions} along $\Gamma_N$.

Since the differential operators in \eqref{math:mixedbcdiff} have closed range, 
we conclude the existence of $C_{{P}} > 0$ 
such that for every $u \in H\Lambda^{k}_T(\Omega)$ 
that is orthogonal to the kernel of \eqref{math:mixedbcdiff:tan}
in the Hilbert space $L^{2}\Lambda^{k}(\Omega)$
we have 
\begin{gather}
 \label{math:poincarefriedrichs}
 \| u \|_{L^{2}\Lambda^{k}(\Omega)}
 \leq
 C_{{P}} \| \cartan u \|_{L^{2}\Lambda^{k+1}(\Omega)}.
\end{gather}
This is a Poincar\'e-Friedrichs inequality, 
and an analogous inequality can be shown for the codifferential \eqref{math:mixedbcdiff:nor}. 

The intersection $H_T\Lambda^{k}(\Omega) \cap H^{\star}_N\Lambda^{k}(\Omega)$
is a Hilbert space when equipped with the canonical intersection scalar product. 
Proposition 4.4 of \cite{GMM} now gives the compactness of the \emph{Rellich embedding}
\begin{align}
 \label{math:compactembedding}
 H_T\Lambda^{k}(\Omega) \cap H^{\star}_N\Lambda^{k}(\Omega) \rightarrow L^{2}\Lambda^{k}(\Omega)
 .
\end{align}
Under stronger assumptions on the domain and the boundary patches 
it is possible to prove the existence of $s \in (0,1]$ and $C > 0$
such that 
\begin{align}
 \label{math:gaffneyinequality}
 \| u \|_{H^{s}\Lambda^{k}(\Omega)}
 \leq
 C \| u \|_{H_T\Lambda^{k}(\Omega) \cap H^{\star}_N\Lambda^{k}(\Omega)},
 \quad 
 u \in H_T\Lambda^{k}(\Omega) \cap H^{\star}_N\Lambda^{k}(\Omega).
\end{align}
Here, $H^{s}\Lambda^{k}(\Omega)$ is the space of differential $k$-forms 
with coefficients in the Bessel potential space $H^{s}(\Omega)$. 
Inequalities of the form \eqref{math:gaffneyinequality} are known as \emph{Gaffney inequalities}
and have been proven under various assumptions on $\Omega$ and $\Gamma_T$.
For example, when $\Gamma_{T}$ is empty, 
then $s = \frac{1}{2}$ holds for every strongly Lipschitz domain 
and $s = 1$ if the domain is even convex \cite{mitrea2001dirichlet}.
We refer to Theorem 4.1 of \cite{GMM}
for the conditions of a Gaffney inequality over strongly Lipschitz domains 
with mixed boundary conditions and $s = \frac{1}{2}$.
If $H^{s}\Lambda^{k}(\Omega)$ is compactly embedded in $L^{2}\Lambda^{k}(\Omega)$, 
then any Gaffney inequality implies the compactness of the Rellich embedding. 

The space of \emph{$k$-th harmonic forms with mixed boundary conditions} is defined as 
\begin{align}
 \label{math:harmonicforms:mixedbc}
 \frakH^{k}(\Omega,\Gamma_T,\Gamma_N)
 :=
 \left\{\;
 p \in H_{T}\Lambda^{k}(\Omega) \cap H^{\star}_{N}\Lambda^{k}(\Omega)
 \suchthat
 \cartan p = 0, \delta p = 0
 \;\right\}
 .
\end{align}
Basic results on Hilbert spaces show that 
\begin{align}
 \label{math:harmonicforms:alternative}
 \frakH^{k}(\Omega,\Gamma_T,\Gamma_N)
 &=
 \left( \ker \cartan: H_T\Lambda^{k}(\Omega) \rightarrow H_T\Lambda^{k+1}(\Omega) \right)
 \cap 
 \left( \cartan H_T\Lambda^{k-1}(\Omega) \right)^{\perp}
 .
\end{align}
Hence we have the $L^{2}$ orthogonal \emph{Hodge decomposition}
\begin{align}
 \label{math:hodgedecomposition}
 L^{2}\Lambda^{k}(\Omega)
 =
 \cartan H_{T}\Lambda^{k-1}(\Omega)
 \oplus
 \frakH^{k}(\Omega,\Gamma_T,\Gamma_N)
 \oplus
 \delta H_{N}^{\star}\Lambda^{k+1}(\Omega)
 .
\end{align}
The dimension of $\frakH^{k}(\Omega,\Gamma_T,\Gamma_N)$ is of particular interest 
because it reflects topological properties of $\Omega$ and $\Gamma_T$.
Specifically, 
by Theorem 5.3 in \cite{GMM} we find that $\frakH^{k}(\Omega,\Gamma_T,\Gamma_N)$ 
is a finite-dimensional space whose dimension 
equals the topological Betti number $b_k\left(\overline\Omega,\Gamma_T\right)$ 
of $\overline\Omega$ relative to $\Gamma_T$.
One can show that 
\begin{align}
 \label{math:bettinumberidentity}
 \dim \frakH^{k}(\Omega,\Gamma_T,\Gamma_N)
 =
 b_{  k}\left(\overline\Omega,\Gamma_T\right)
 =
 b_{n-k}\left(\overline\Omega,\Gamma_N\right)
 ,
 \quad
 0 \leq k \leq n.
\end{align}
In the special cases $\Gamma_T = \emptyset$ and $\Gamma_T = \partial\Omega$,
which have received most of the attention in the literature,
the Betti numbers correspond to the topological properties of the domain only,
such as the number of connected components or of holes of a certain dimension.
But in the presence of mixed boundary conditions  
the Betti numbers depend also on the topology of the boundary patch $\Gamma_T$.

\begin{example}
 \label{example:bettinumbers}
 The spaces $\frakH^{0}(\Omega,\Gamma_T,\Gamma_N)$
 and $\frakH^{n}(\Omega,\Gamma_T,\Gamma_N)$
 are spanned by the locally constant functions over $\Omega$
 whose supports are disjoint from $\Gamma_T$ and $\Gamma_N$, respectively.
 The other harmonic spaces have more complicated descriptions,
 but their dimensions are often easier to determine.
 
 For example, if $\Omega = (-1,1)^{2}$ and $\Gamma_T$ has $M \in \bbN$ connected components, 
 then $b_{1}(\overline\Omega,\Gamma_T) = M - 1$. 
 In the specific case $\Gamma_T = \{0,1\} \times (-1,1)$
 we have $b_{1}(\overline\Omega,\Gamma_T) = 1$
 and $\frakH^{1}(\Omega,\Gamma_T,\Gamma_N)$ has a very simple description:
 it corresponds to the span of the constant vector field taking the value $(1,0)$
 over all of $\Omega$.
\end{example}

The $k$-th \emph{Hodge Laplacian} is the unbounded operator 
\begin{align*}
 \Laplace_{k} : 
 \dom(\Laplace_{k}) \subseteq L^{2}\Lambda^{k}(\Omega) \rightarrow L^{2}\Lambda^{k}(\Omega),
 \quad 
 u \mapsto \delta \cartan u + \cartan \delta u
\end{align*}
whose domain is defined as 
\begin{align*}
 \dom(\Laplace_{k})
 :=
 \left\{\;
  u \in H_{T}\Lambda^{k}(\Omega) \cap H_{N}^{\star}\Lambda^{k}(\Omega)
  \suchthat 
  \cartan u \in H_{N}^{\star}\Lambda^{k+1}(\Omega),
  \delta u \in H_{T}\Lambda^{k-1}(\Omega)
 \;\right\}
 .
\end{align*}
One can show \cite[Theorem 4.5]{GMM} that $\Laplace_{k}$ is densely-defined, closed, 
and self-adjoint with closed range, and moreover that 
\begin{align*}
 \ker \Laplace_{k} = \frakH^{k}(\Omega,\Gamma_T,\Gamma_N),
 \quad 
 \rng \Laplace_{k} = \frakH^{k}(\Omega,\Gamma_T,\Gamma_N)^{\perp}.
\end{align*}
This implies the existence of a bounded \emph{solution operator}
$G_{k} : L^{2}\Lambda^{k}(\Omega) \rightarrow L^{2}\Lambda^{k}(\Omega)$,
the pseudoinverse of $\Delta_{k}$,
which maps into $\dom(\Delta_{k})$ and satisfies
\begin{gather*}
 \rng G_{k} = \frakH^{k}(\Omega,\Gamma_T,\Gamma_N)^{\perp}
 ,
 \quad 
 \ker G_{k} = \frakH^{k}(\Omega,\Gamma_T,\Gamma_N)
 ,
 \\
 \forall f \in L^{2}\Lambda^{k}(\Omega) \cap \frakH^{k}(\Omega,\Gamma_T,\Gamma_N)^{\perp} 
 : 
 f = \Laplace_{k} G_{k} f
 ,
 \\
 \forall 
 u \in \dom(\Laplace_{k}) \cap \frakH^{k}(\Omega,\Gamma_T,\Gamma_N)^{\perp}
 :
 u = G_{k} \Laplace_{k} u
 .
\end{gather*}
The $k$-th \emph{Hodge Laplace equation} with mixed boundary conditions
for a given right-hand side $f \in L^{2}\Lambda^{k}(\Omega)$
is the partial differential equation
\begin{align}
 \label{math:hodgelaplaceequation}
 \Laplace_{k} u = f
\end{align}
in the unknown $u \in \dom\left( \Laplace_{k} \right)$.
In general, we can solve the $k$-th Hodge Laplace equation only in the sense of least squares 
whenever there exist non-trivial $k$-th harmonic forms;
the least-squares solution is precisely $u = G_{k} f$.
Any solution of \eqref{math:hodgelaplaceequation} and its derivatives 
satisfies the tangential and normal boundary conditions 
that are encoded in the definition of $\dom(\Delta_{k})$.

\subsection{Variational Theory and Finite Element Approximation}
\label{subsec:finiteelementapproximation}

The Hodge Laplacian is self-adjoint with closed range.
Can we develop a finite element method for the Hodge Laplace equation 
that minimizes the energy functional 
\begin{align*}
 \calJ( v ) 
 := 
 \onehalf \int_{\Omega} 
 \left| \cartan v \right|^{2} + \left| \delta v \right|^{2}
 \;\dif x
 - 
 \int_{\Omega} 
 \langle f, v \rangle
 \;\dif x
\end{align*}
over a subspace of $H_T\Lambda^{k}(\Omega) \cap H^{\star}_N\Lambda^{k}(\Omega)$?
While for $k=0$ this is just the canonical approach for the Poisson problem (already mentioned in the introduction),
it is problematic when $k > 0$.
To begin with, the variational theory of the energy functional involves 
a Lagrange multiplier in $\frakH^{k}(\Omega,\Gamma_T,\Gamma_N)$,
for which we only have non-conforming approximations in practice. 
A more severe difficulty is that 
piecewise polynomial differential forms generally fail to approximate members of $H_T\Lambda^{k}(\Omega) \cap H^{\star}_N\Lambda^{k}(\Omega)$
in the canonical intersection norm of that space,  
which is why a finite element method based on minimizing $\calJ$ is generally inconsistent \cite{costabel1991coercive}.

We will circumvent these difficulties by introducing $\sigma = \delta u$ as an auxiliary variable
and reformulating the Hodge Laplace problem in a saddle point formulation,
following \cite{AFW2}. 
This approach leads to the system 
\begin{subequations}
\label{math:mixedproblem}
\begin{align}
  \langle \sigma, \tau \rangle_{L^{2}} - \langle u, \cartan \tau \rangle_{L^{2}}
  &= 0,
  &{}&
  \tau \in H_T\Lambda^{k-1}(\Omega),
  \\
  \langle \cartan \sigma, v \rangle_{L^{2}} + \langle \cartan u, \cartan v \rangle_{L^{2}} + \langle p, v \rangle_{L^{2}}
  &= \langle f, v \rangle_{L^{2}},
  &{}&
  v \in H_T\Lambda^{k  }(\Omega),
  \\
  \langle u, q \rangle_{L^{2}}  
  &= 0,
  &{}&
  q \in \frakH^{k}(\Omega,\Gamma_T,\Gamma_N),  
\end{align}
\end{subequations}
where $\sigma \in H_T\Lambda^{k-1}(\Omega)$, $u \in H_T\Lambda^{k}(\Omega)$, and $p \in \frakH^{k}(\Omega,\Gamma_T,\Gamma_N)$ are the unknowns.
These are the Euler-Lagrange equations of a saddle point functional.
One can show \cite[Theorem 3.1]{AFW2} that \eqref{math:mixedproblem}
has a unique solution for every $f \in L^{2}\Lambda^{k}(\Omega)$.
Furthermore, by the discussion in \cite[Subsection 3.2.1]{AFW2}
we see that \eqref{math:mixedproblem} is equivalent to \eqref{math:hodgelaplaceequation}
in the following sense:
the tuple $(\sigma,u,p)$ solves \eqref{math:mixedproblem}
if and only if 
\begin{align*}
 u = G_{k} f, \quad p = f - G_{k} f, \quad \sigma = \delta u.
\end{align*}
Moreover, the norm of the solution $(\sigma,u,p)$ 
in $H_T\Lambda^{k-1}(\Omega) \times H_T\Lambda^{k}(\Omega) \times L^{2}(\Omega)$
is uniformly bounded in $L^{2}$ norm of $f$.
Hence the saddle point problem is well-posed.

To construct a finite element method based on \eqref{math:mixedproblem}
we consider a 
finite element de~Rham complex as a discretization 
of the $L^{2}$ de~Rham complex \eqref{math:derhamcomplexmixedbc:tangential} with tangential boundary conditions. 
Specifically, let $\calT$ be a simplicial complex that triangulates $\overline\Omega$ 
and contains a subtriangulations $\calU$ of $\Gamma_T$.
As in Section~\ref{sec:projection} 
we fix a finite element de~Rham complex with partial boundary conditions:
\begin{align}
 \label{math:finiteelementderhamcomplex}
 \begin{CD}
  0 
  @>>>
  \calP\Lambda^{0}(\calT,\calU)
  @>\cartan>>
  \dots
  @>\cartan>>
  \calP\Lambda^{n}(\calT,\calU)
  @>>>
  0
  .
 \end{CD}
\end{align}
Theorem~\ref{prop:finalprojection} gives a bounded projection 
$\pi^{k} : L^{2}\Lambda^{k}(\Omega) \rightarrow \calP\Lambda^{k}(\calT,\calU)$ 
whose operator norm depends only on 
the polynomial degree of the finite element spaces, the mesh quality, and the geometry of $\Omega$,
and for which we have the following commuting diagram:
\begin{gather*}
 \begin{CD}
  0 
  @>>>
  H_T\Lambda^{0}(\calT)
  @>\cartan>>
  \dots
  @>\cartan>>
  H_T\Lambda^{n}(\calT)
  @>>>
  0
  \\
  @.
  @V{\pi^{0}}VV
  @.
  @V{\pi^{n}}VV
  @.
  \\
  0 
  @>>>
  \calP\Lambda^{0}(\calT,\calU)
  @>\cartan>>
  \dots
  @>\cartan>>
  \calP\Lambda^{n}(\calT,\calU)
  @>>>
  0
  .
 \end{CD}
\end{gather*}
A central concept on the analytical level which we want to mimic on the discrete level
are the harmonic forms.
We define the \emph{$k$-th discrete harmonic space} as
\begin{align*}
 \label{math:discreteharmonicspace}
 \frakH^{k}(\calT,\calU) 
 :=
 \left( \ker \cartan: \calP\Lambda^{k}(\calT,\calU) \rightarrow \calP\Lambda^{k+1}(\calT,\calU) \right)
 \cap 
 \left( \cartan \calP\Lambda^{k-1}(\calT,\calU) \right)^{\perp}
 .
\end{align*}
We note that this definition of discrete harmonic $k$-forms $\frakH^{k}(\calT,\calU)$ 
is entirely analogous to the identity satisfied 
by the harmonic $k$-forms $\frakH^{k}(\overline\Omega,\Gamma_T,\Gamma_N)$. 
The dimension of $\frakH^{k}(\calT,\calU)$
is the Betti number $b_k(\overline\Omega,\Gamma_T)$ of $\overline\Omega$ relative to $\Gamma_T$,
as follows, e.g., by Corollary~2 in \cite{licht2016discrete},
but can also be shown with an adaption of methods in \cite{AFW2}.
In particular, the dimension of $\frakH^{k}(\calT,\calU)$
depends only on $\Omega$ and $\Gamma_T$. 

We outline a mixed finite element method for the saddle point system \eqref{math:mixedproblem}.
We search for $\sigma_h \in \calP\Lambda^{k}(\calT,\calU)$, 
$u_h \in \calP\Lambda^{k}(\calT,\calU)$, 
and $p_h \in \frakH^{k}(\calT,\calU)$
such that 
\begin{subequations}
\label{math:discretemixedproblem}
 \begin{align}
  \langle \sigma_h, \tau_h \rangle_{L^{2}} - \langle u_h, \cartan \tau_h \rangle_{L^{2}}
  &= 0,
  &{}&
  \tau_h \in \calP\Lambda^{k-1}(\calT,\calU),
  \\
  \langle \cartan \sigma_h, v_h \rangle_{L^{2}} + \langle \cartan u_h, \cartan v_h \rangle_{L^{2}} + \langle p_h, v_h \rangle_{L^{2}}
  &= \langle f, v_h \rangle_{L^{2}},
  &{}&
  v_h \in \calP\Lambda^{k}(\calT,\calU),
  \\
  \langle u_h, q_h \rangle_{L^{2}}  
  &= 0,
  &{}&
  q_h \in \frakH^{k}(\calT,\calU).  
\end{align}
\end{subequations}
The existence of a uniformly bounded smoothed projection enables the Galerkin theory of Hilbert complexes.
We find that discrete problem is well-posed: 
there exists uniformly bounded $C > 0$ such that \eqref{math:discretemixedproblem}
has a unique solution $(\sigma_{h},u_{h},p_{h})$ that satisfies 
\begin{gather*}
 \| \sigma_{h} \|_{H\Lambda^{k-1}(\Omega)}
 +
 \| u_{h} \|_{H\Lambda^{k}(\Omega)}
 +
 \| p_{h} \|_{L^{2}\Lambda^{k}(\Omega)}
 \leq 
 C
 \| f \|_{L^{2}\Lambda^{k}(\Omega)}
 .
\end{gather*}
Furthermore, a priori error estimates can be shown as in Theorem~3.9 of \cite{AFW2}. 
We define an auxiliary quantity $\mu$ that measures the approximation of harmonic forms,
\begin{gather*}
 \mu 
 :=
 \sup_{ \substack{ p \in \frakH^{k}(\Omega,\Gamma_T,\Gamma_N) \setminus \{0\} \\ p \neq 0 } } 
 \dfrac{ 
  \| p - \pi^{k}_{h} p \|_{L^{2}\Lambda^{k}(\Omega)}
 }{
  \| p \|_{L^{2}\Lambda^{k}(\Omega)}
 }
 .
\end{gather*}
Letting $(\sigma,u,p)$ denote the solution of the original system \eqref{math:mixedproblem},
there exists $C > 0$ uniformly bounded such that 
\begin{align*}
 &
 \| \sigma - \sigma_{h} \|_{H\Lambda^{k-1}(\Omega)}
 +
 \| u - u_{h} \|_{H\Lambda^{k}(\Omega)}
 +
 \| p - p_{h} \|_{L^{2}\Lambda^{k}(\Omega)}
 \\
 &
 \leq 
 C
 \bigg(
 \inf_{ \tau_{h} \in \calP\Lambda^{k-1}(\calT,\calU) }
 \| \sigma - \tau_{h} \|_{H\Lambda^{k-1}(\Omega)}
 +
 \inf_{ v_{h} \in \calP\Lambda^{k}(\calT,\calU) }
 \| u - v_{h} \|_{H\Lambda^{k}(\Omega)}
 \\&\qquad\qquad
 +
 \inf_{ q_{h} \in \calP\Lambda^{k}(\calT,\calU) }
 \| p - q_{h} \|_{L^{2}\Lambda^{k}(\Omega)}
 +
 \mu
 \inf_{ v_{h} \in \calP\Lambda^{k}(\calT,\calU) }
 \| u_{\cartan} - v_{h} \|_{H\Lambda^{k}(\Omega)}
 \bigg)
 ,
\end{align*}
where $u_{\cartan}$ denotes the $L^{2}$ orthogonal projection 
of $u$ onto $\cartan H_T\Lambda^{k-1}(\Omega)$.
More specific error estimates for the $L^{2}$ norm 
of $u$ and its derivatives, generalizing Aubin-Nitsche-type techniques,
are described in Theorem~3.11 of \cite{AFW2}.
These rely on the compactness of the solution operator $G$,  
which is a direct consequence of compactness of the Rellich embedding \eqref{math:compactembedding}. 
Gaffney inequalities such as \eqref{math:gaffneyinequality} are important 
for deriving convergence rates in terms of mesh size parameters. 
The reader is referred to Section~3 of \cite{AFW2} for all details.

\subsection{Examples in Vector Analysis}
\label{subsec:examplesinvectoranalysis}

For the purpose of demonstration, 
we consider special cases of the Hodge Laplace equation
in three dimensions in the notation of classical vector calculus here.
Here we assume that $\Omega$ is a strongly Lipschitz domain
and let $\vec{n} : \partial\Omega \rightarrow \bbR^{3}$
be the outward unit normal field along the boundary.

\subsubsection{The case $k=0$}
This is the Poisson problem with mixed boundary conditions.
The boundary part $\Gamma_T$ corresponds to the \emph{Dirichlet boundary part},
and $\Gamma_N$ corresponds to the \emph{Neumann boundary part}.
We let $\calH^{0}$ be the span of the indicator functions 
of those connected components of $\Omega$ that do not touch $\Gamma_T$. 
The Hodge Laplace problem is to find $u \in H^{1}(\Omega)$
and $p \in \calH^{0}$ such that 
\begin{gather*}
 -\divergence\grad u + p = f, \quad u \perp \calH^{0},
 \quad 
 u_{|\Gamma_T} = 0, \quad (\grad u)_{|\Gamma_N} \cdot\vecn = 0
\end{gather*}
for given $f \in L^{2}(\Omega)$.
The condition $u \perp \calH^{0}$ enforces that $u$
has vanishing mean on those components of $\Omega$ that touch $\Gamma_T$. 
The corresponding discretization is the primal method for the Poisson problem.

\subsubsection{The case $k=1$}
In this case, the Hodge Laplace operator translates to the vector Laplace operator.
The harmonic $1$-forms correspond to the vector fields 
\begin{gather*}
 \vec \calH_1 =
 \left\{\;
  \vec p \in L^2(\Omega,\bbR^{3})
  \suchthat
  \divergence \vec p = 0, \;
  \curl \vec p = 0, \;
  \vec p_{|\Gamma_N} \cdot \vecn = 0, \;
  \vec p_{|\Gamma_T} \times \vecn = 0
 \;\right\}
 .
\end{gather*}
The Hodge Laplace problem translates as follows.
Given $\vec f \in L^2(\Omega,\bbR^{3})$
we seek $\vec u \in H(\Omega,\curl)$, $\sigma \in H^{1}(\Omega)$,
and $\vec p \in \vec \calH^{1}$
such that 
\begin{gather*}
 \sigma = - \divergence \vec u,
 \quad
 \grad\sigma + \curl\curl \vec u = \vec f - \vec p,
 \quad
 \vec u \perp \vec \calH_1
 ,\\
 \sigma_{|\Gamma_T}   = 0,
 \quad
 \vec u_{|\Gamma_T} \times \vecn = 0,
 \quad
 \vec u_{|\Gamma_N} \cdot \vecn = 0, 
 \quad 
 (\curl \vec u)_{|\Gamma_N} \times \vecn = 0.
\end{gather*}
The divergence of $\vec u$ is treated as an auxiliary variable.
In FEEC, 
the $\curl$-$\curl$-subsystem of the system is discretized with a primal method,
and the $\grad$-$\divergence$-subsystem is treated with a mixed discretization.

\subsubsection{The case $k=2$}
This is another formulation of the vector Laplacian. 
In comparison to the case $k=1$,
the roles of $\Gamma_T$ and $\Gamma_N$ are reversed.
The harmonic $2$-forms correspond to the space  
\begin{gather*}
 \vec \calH_2 
 =
 \left\{\;
  \vec p \in L^2(\Omega,\bbR^{3})
  \suchthat
  \divergence \vec p = 0, \; \curl \vec p = 0, \;
  \vec p_{|\Gamma_N} \times \vecn = 0, \; \vec p_{|\Gamma_T} \cdot \vecn = 0
 \;\right\}
 .
\end{gather*}
The Hodge Laplace equation translates as follows.
Given $\vec f \in L^{2}(\Omega,\bbR^{3})$,
we seek $\vec u \in H(\Omega,\divergence)$, $\vec\sigma \in H(\Omega,\curl)$,
and $p \in \vec\calH^{2}$ 
which satisfy
\begin{gather*}
 \vec \sigma = \curl \vec u, \quad
 \curl \vec \sigma - \grad\divergence \vec u = \vec f - \vec p, \quad
 \vec u \perp \vec \calH_2 
 ,\\
 \vec u_{|\Gamma_T} \cdot \vecn = 0,
 \quad
 \sigma_{|\Gamma_T} \times \vecn = 0,
 \quad
 \vec u_{|\Gamma_N} \times \vecn = 0, 
 \quad
 ( \divergence \vec u )_{|\Gamma_N} = 0.
\end{gather*}
%
In comparison to the case $k=1$, the role of essential and natural boundary conditions is reversed. 
In FEEC, the $\grad$-$\divergence$-subsystem is discretized in primal formulation,
while the $\curl$-$\curl$-subsystem is discretized in mixed formulation.

\subsubsection{The case $k=3$}
This is again the Poisson problem with mixed boundary conditions,
but with the roles of $\Gamma_T$ and $\Gamma_N$ reversed in comparison to the case $k=0$.
We let $\calH_3$ be spanned by the indicator functions 
of those connected components of $\Omega$ that do not touch $\Gamma_{N}$. 
For some given $f \in L^{2}(\Omega)$, 
the Hodge Laplace problem reduces to finding  
$u \in L^{2}(\Omega)$, $\sigma \in H(\Omega,\divergence)$, and $p \in \calH^{3}$
such that  
\begin{gather*}
 \vec \sigma = -\grad u,
 \quad
 \divergence \vec \sigma = f - p, 
 \quad
 u \perp \mathfrak \calH_3, 
 \quad 
 \sigma_{|\Gamma_T} \cdot \vecn = 0, \quad u_{|\Gamma_N} = 0.
\end{gather*}
The orthogonality condition forces
$u$ to have vanishing mean on the connected components of $\Omega$ that touch $\Gamma_N$.
The role of essential and natural boundary conditions is reversed
in comparison to the case $k=0$. 
The corresponding discretization is a mixed finite element method for the Poisson problem.


\bibliographystyle{amsplain}
\bibliography{references.mixedbc}

\end{document}